\documentclass[11pt]{amsart}

\usepackage{xspace,amssymb,amsmath,amscd,amsthm,epsfig,fullpage,booktabs,color}
\usepackage{xcolor}
\definecolor{myurlcolor}{RGB}{0,0,139} 
\definecolor{mycitecolor}{HTML}{006600} 
\usepackage{hyperref} 

\definecolor{myurlcolor}{HTML}{a6cee3}

\definecolor{greyblue}{rgb}{0.357,0.447,0.71}
\definecolor{mylinkcolor}{rgb}{0.71,0.447,0.357}
\definecolor{mycitecolor}{rgb}{0.357,0.447,0.71}

\hypersetup{
    colorlinks=true,
    linkcolor=mylinkcolor,  
    filecolor=magenta,   
    urlcolor=myurlcolor, 
    citecolor=mycitecolor 
}

\newcommand{\bx}{\boldsymbol{x}}
\newcommand{\compl}[1]{\overline{#1}}

\newcommand{\N}{\text{N}}
\newcommand{\E}{\text{E}}
\DeclareMathOperator{\neigh}{N}
\DeclareMathOperator{\Cat}{Cat}

\DeclareMathOperator{\lvl}{lvl}

\newcommand{\ones}{\mathbf{1}}
\newcommand{\ee}{\mathbf{e}}
\newcommand{\vv}{\mathbf{v}}
\newcommand{\ww}{\mathbf{w}}
\newcommand{\Ga}{\Gamma}

\newcommand{\degch}[1]{\left|#1\right|}
\newcommand{\mvl}[1]{\tilde{#1}}

\makeatletter
\newfont{\footsc}{cmcsc10 at 8truept}
\newfont{\footbf}{cmbx10 at 8truept}
\newfont{\footrm}{cmr10 at 10truept}
\renewcommand{\ps@plain}{%
\renewcommand{\@oddfoot}{\footsc 
  \hfil\footrm\thepage}}
\makeatother
\theoremstyle{plain}
\newtheorem{theorem}{Theorem}
\newtheorem{lemma}[theorem]{Lemma}
\newtheorem{prop}[theorem]{Proposition}

\theoremstyle{definition}
\newtheorem{remark}[theorem]{Remark}

\newtheorem{example}[theorem]{Example}

\newcommand{\mcG}{\mathcal{G}}
\newcommand{\mcLP}{\mathcal{LP}}
\newcommand{\mcP}{\mathcal{P}}

\newcommand{\lpath}{\textsc{lpath}}
\newcommand{\upath}{\textsc{path}}
\DeclareMathOperator{\pos}{pos}
\DeclareMathOperator{\run}{run}

\newcommand{\Z}{\mathbb{Z}}
\newcommand{\R}{\mathbb{R}}

\newcommand{\ch}{D}



\begin{document} 

\title{Quantized rational chip-firing}

\thanks{This work was supported by Gifts from the Simons Foundation/SFARI
(\#854037 to Backman; \#633564 to Loehr; \#429570 to Warrington) and NSF Grant DMS-2246967 (Backman).}

\author[S. Backman et al.]{Spencer Backman}
\address{Dept. of Mathematics and Statistics, University of Vermont, Burlington, VT 05401}
\email{spencer.backman@uvm.edu}

\author[]{Nicholas A. Loehr}
\address{Dept. of Mathematics, Virginia Tech, Blacksburg, VA 24061 }
\email{nloehr@vt.edu}

\author[]{Gregory S. Warrington}
\address{Dept. of Mathematics and Statistics, University of Vermont, Burlington, VT 05401}
\email{gregory.warrington@uvm.edu}

\begin{abstract}
 This article introduces a quantized chip-firing model with close
 connections to the theory of rational lattice paths and rational
 parking functions. Given a graph with a sink and positive integers 
 $a,b,c$ with $\gcd(a,b)=1$, a set $S$ of vertices fires by the following rule.
 Each vertex in $S$ provisionally sends $c$ chips to the sink and
 $a/b$ chips to each non-sink neighbor outside of $S$.  The novel
 feature is that the total number of chips leaving from or arriving at
 any vertex gets rounded down to the nearest integer before being
 finalized. We define the notions of chip configurations being
 superstable, $k$-stable, or $k$-skeletal in this model.  When $c=1$
 and the graph is complete, superstable configurations correspond to
 rational parking functions. There is a bijection between superstable
 configurations and $k$-skeletal configurations for each $k$. We
 establish these results by building a combinatorial theory of
 $k$-skeletal rational lattice paths (both unlabeled and labeled) and
 translating that theory to chip configurations.  There is a group
 structure on the set of chip configurations modulo firing and
 borrowing moves. We show that this group is isomorphic to the product
 of $b-1$ copies of the integers modulo $a$; and, for each $k$, each
 coset of chip configurations in this group contains a unique
 $k$-skeletal representative.
\end{abstract}

\maketitle

\section{Introduction}
\label{sec:intro}

The goal of this article is to introduce a new chip-firing model,
called \emph{quantized rational chip-firing}, that makes close contact
with the combinatorics of labeled and unlabeled rational-slope lattice
paths. In generalizing the classic model, we extend chip-firing to
graphs with rational edge weights while extending a number of
important properties such as duality between superstable and recurrent
configurations (\S\ref{subsec:duality};
cf.~\cite{benton2024zsuperstable,cho-siam}), enumerative results
relating to the Catalan numbers and parking functions (Facts~(F3)
and~(F4) in~\S\ref{subsec:rat-chip-model1}), and the existence of an
associated critical group (\S\ref{sec:group}).\footnote{While we focus on the
case of a uniform, rational edge weight between all non-sink
vertices, our framework easily handles the case of arbitrary, real
edge weights; see Remark~\ref{rem:real}.}

In this introduction, we first present a standard integral chip-firing
model in \S\ref{subsec:int-chip-model} before outlining a simplified
version of our quantized rational chip-firing model in
\S\ref{subsec:rat-chip-model1}. Sections~\ref{subsec:kskel-intro} and
~\ref{subsec:kskel-path-intro} then briefly present versions of the
$k$-skeletal objects from ~\cite{ratchip1}
(cf.~\cite{Backman-bij,caracciolo2012multiple}) that are associated
with our new model. In these sections we also summarize our results
relating chip configurations to rational lattice paths.  In
\S\ref{subsec:group-intro} we describe a group structure arising from
our quantized chip-firing model.  Finally, in \S\ref{subsec:related}
we provide more context for these results by comparing them to some of
the relevant literature.

\subsection{An Integral Chip-Firing Model}
\label{subsec:int-chip-model}

This chip-firing model (see~\cite{bak,Dhar,klivans}) starts with
positive integer parameters $m,n,c$.  Let $[n]=\{1,2,\ldots,n\}$. Let
$G$ be the complete graph on vertex set $\{0\}\cup [n]$, where $0$ is
a special vertex called the \emph{sink}.  A \emph{chip configuration}
on $G$ is a function $\ch:[n]\rightarrow\Z$ where $\ch(i)$ represents
the number of chips at vertex $i$ ($\ch$ for ``divisor'').  We
may write $\ch$ as a list $(\ch(1),\ch(2),\ldots,\ch(n))$ or as a word
$\ch(1)\ch(2)\cdots\ch(n)$.  We say $\ch$ is \emph{nonnegative},
written $\ch\geq 0$, if $\ch(i)\geq 0$ for all $i\in [n]$.  The chip
count at the sink is not recorded in $\ch$.

We think of $c$ as the capacity of each edge between the sink $0$
to any other vertex in $G$, while $m$ is the capacity of each edge
linking two non-sink vertices. A single vertex $i\in [n]$ \emph{fires}
by sending $c$ chips to the sink and $m$ chips to each other vertex in $[n]$.
More generally, for any $s$-element subset $S$ of $[n]$, 
the \emph{cluster-firing move}
$\phi_S$ acts on any configuration $\ch$ as follows.
Write $\compl{S}=[n]\setminus S=\{j\in [n]: j\not\in S\}$.
For each $i\in S$, replace $\ch(i)$ by $\ch(i)-(c+m(n-s))$.
For each $j\in\compl{S}$, replace $\ch(j)$ by $\ch(j)+ms$.
Informally, each $i\in S$ sends $c$ chips to the sink and $m$ chips
to each $j\in\compl{S}$. For $\ch\geq 0$, we say $S$ \emph{can legally fire}
in configuration $\ch$ if $\phi_S(\ch)\geq 0$, which means
$\ch(i)\geq c+m(n-s)$ for all $i\in S$. Throughout this paper,
the unqualified phrase ``can fire'' always abbreviates ``can legally fire.''
For $S=\varnothing$, $\phi_{\varnothing}$
is the identity map. A \emph{$k$-firing move} is a cluster-fire move
$\phi_S$ where $0<|S|\leq k+1$.

A nonnegative chip configuration $\ch$ is \emph{stable} if no single vertex
can fire, while $\ch$ is \emph{superstable} if no nonempty subset can fire.
More generally, for each $k$ in the range $0\leq k<n$, $\ch$ is called
\emph{$k$-stable} if no $k$-firing move is legal for $\ch$.
Any $\ch\geq 0$ can be converted to a $k$-stable configuration
by a finite sequence of legal $k$-firing moves.
In this integral case, it can be shown that the resulting $k$-stable
configuration is uniquely determined by $\ch$; it is called
\emph{the $k$-stabilization} of $\ch$.  The proof relies on
a version of the Diamond Property~\cite[Thm. 2.2.2(1)]{klivans} or
Newman's Lemma~\cite{newman}; also see
Remark~\ref{rem:int-kstab-uniq}.
It is also true that superstable configurations correspond naturally to 
certain labeled lattice paths that proceed by unit-length north steps
and east steps from $(0,0)$ to $(mn+c-1,n)$ while staying weakly inside
the triangle or trapezoid with vertices $(0,0)$, $(0,n)$, $(mn+c-1,n)$, 
and $(c-1,0)$.
The symmetric group $S_n$ acts on chip configurations by permuting
the non-sink vertices. The number of $S_n$-orbits of superstable
chip configurations is the number of unlabeled lattice paths contained
in this shape. When $c=1$, this number is the Fuss--Catalan number
$\binom{1}{mn+1}\binom{mn+n}{n}$.

\subsection{Quantized Rational Chip-Firing Model for Complete Graphs}
\label{subsec:rat-chip-model1}

Our quantized rational chip-firing model can be defined for an
arbitrary graph $G$ and integer parameters $a,b,c\in
\mathbb{Z}_{\geq 0}$. This general version is presented in
\S\ref{sec:rat-chip-model}. In this section, we 
restrict ourselves to the main focus of this article, which is
this important special case: $a$ and $b$ are relatively prime, $c=1$, and
$G$ is the complete graph with $b+1$ vertices. For the vertex set of
$G$, we take the set $\{0\}\cup [b]$, where vertex $0$ is the sink.
A \emph{chip configuration} on $G$
is a function $\ch:[b]\rightarrow\Z$; we define $\ch\geq 0$ as before.

Let $S$ be any $s$-element subset of $[b]$.
The \emph{cluster-fire move} $\phi_S$ acts on a configuration $\ch$
as follows. For all $i\in S$, subtract $1+\lfloor (b-s)a/b\rfloor$
from $\ch(i)$. For all $j\in\compl{S}=[b]\setminus S$, 
add $\lfloor sa/b\rfloor$ to $\ch(j)$.
We can understand this move intuitively as a two-step process.
First, each vertex in $S$ sends 1 chip to the sink and 
$a/b$ chips to every other vertex not in $S$. Second, the total
chip gain or chip loss at each vertex (always viewed as a positive number)
is rounded \emph{down} to the nearest integer. 

We say $S$ \emph{can (legally) fire} in configuration
$\ch\geq 0$ if $\phi_S(\ch)\geq 0$. Because of the rounding rules,
cluster-firing $S$ is not equivalent (in general) to sequentially firing
each individual vertex in $S$. As before, a \emph{$k$-firing move}
is a cluster-fire move $\phi_S$ where $0<|S|\leq k+1$.
Stable, superstable, and $k$-stable configurations are defined as
in the previous model. 

\begin{remark}\label{rem:basic-bound}
For an individual vertex to fire legally in this model, 
the chip count at that vertex must be at least 
$1+\lfloor(b-1)a/b\rfloor=1+a+\lfloor-a/b\rfloor$.
This quantity never exceeds $a$. Therefore, a necessary condition
for a configuration $\ch$ to be $k$-stable is that
$0\leq \ch(i)<a$ for all $i\in [b]$. Lemma~\ref{lem:test-stable}
furnishes a necessary and sufficient condition for $k$-stability.
\end{remark}

Let $\ch$ be a nonnegative chip configuration for the chip-firing
model described in this subsection, and let $0\leq k<b$. We collect
here a number of facts about quantized rational chip-firing in the
context of the complete graph. The last two facts
establish enumerative connections between this chip-firing model and
labeled and unlabeled \emph{rational-slope lattice paths} that are
contained in the triangle with vertices $(0,0)$, $(0,b)$, and
$(a,b)$. These connections were our primary motivation for introducing
this new chip-firing model.

\begin{enumerate}\label{thm:ratchip-facts}
\item[(F1)] There exists a $k$-stable configuration reachable from $\ch$ by
  a finite sequence of legal $k$-firing moves. We call any such
  configuration a \emph{$k$-stabilization} of
  $\ch$. (See Proposition~\ref{prop:exist-kstab}.)\label{thm:ratchip-facts:stab}
\item[(F2)] There is a unique $k$-stabilization of $\ch$ for $k=0$ and for
  $k=b-1$, but uniqueness is not guaranteed for other values of
  $k$. (See Theorem~\ref{thm:unique0} and
  Proposition~\ref{prop:ss-uniq}. Note that the uniqueness for $k=b-1$ is
  only true for the complete graph.)\label{thm:ratchip-facts:uniq}
\item[(F3)] There are $a^{b-1}$ superstable configurations. (See
  Theorem~\ref{thm:kskel-chip}\ref{thm:kskel-chip:ab-1}.)\label{thm:ratchip-facts:ab-1}
\item[(F4)] Letting $S_b$ act on chip configurations by permuting the
  non-sink vertices, there are $\frac{1}{a+b}\binom{a+b}{a,b}$ orbits
  of superstable configurations. (See
  Theorem~\ref{thm:kskel-chip}\ref{thm:kskel-chip:cat}.)\label{thm:ratchip-facts:cat}
\end{enumerate}

Section~\ref{sec:rat-chip-model} presents our general quantized
rational chip-firing model where $G$ need not be complete, $a$ and $b$
need not be coprime, and sink edges need not have capacity $1$. In
that setting, we prove $k$-stabilizations exist for all $k$ but are
only guaranteed to be unique for $k=0$.

\begin{example}
 Take initial parameters $a=5$ and $b=3$. 
 Firing one vertex means that vertex loses $1+\lfloor (3-1)5/3\rfloor
    =4$ chips and the other two non-sink
    vertices gain $\lfloor 1\cdot 5/3\rfloor = 1$ chip each.
 Firing a set of two vertices means those vertices lose 
    $1+\lfloor (3-2)5/3\rfloor =2$
     chips each and the other non-sink vertex gains 
    $\lfloor 2\cdot 5/3\rfloor = 3$ chips.
  Firing all three vertices means all non-sink vertices lose $1+\lfloor
    (3-3)5/3\rfloor = 1$ chip. 

We compute a $k$-stabilization of $\ch=555$ for $k=0,1,2$.
The vertices that cluster-fire to reach the next configuration
are indicated by underlining.

\begin{itemize}
\item $0$-stabilization:
 $55\underline{5} \rightarrow \underline{6}61 \rightarrow 2\underline{7}2 
\rightarrow 333$.
\item $1$-stabilization:
 $55\underline{5} \rightarrow \cdots \rightarrow \underline{33}3 
\rightarrow 11\underline{6} \rightarrow \underline{22}2 
\rightarrow 00\underline{5} \rightarrow 111$.
\item $2$-stabilization:
 $55\underline{5} \rightarrow \cdots \rightarrow \underline{111} 
 \rightarrow 000.$
\end{itemize}
\end{example}

\begin{example}\label{ex:ex75a}
To see the potential non-uniqueness of $k$-stabilizations, take $a=7$,
$b=5$, and $k=2$. For some nonempty $k$-firing move to be legal on
$\ch$, there must be one vertex in $\ch$ with at least $6$ chips, or
two vertices with at least $5$ chips each, or three vertices with at
least $3$ chips each.  Starting with $\ch=00355$, we find that
$\phi_{\{4,5\}}(\ch)=22500$ and $\phi_{\{3,4,5\}}(\ch)=44022$ are two
different $2$-stabilizations of $\ch$. 
One may check that there are no other $2$-stabilizations of $\ch$.
\end{example}

\subsection{\texorpdfstring{$k$}{k}-Skeletal Chip Configurations}
\label{subsec:kskel-intro}

To remedy the non-uniqueness of $k$-stabilizations, we introduce
the related concept of $k$-skeletal configurations.
This concept was first studied in~\cite{ratchip1} for the 
model in~\S\ref{subsec:int-chip-model} and its generalization to
the case where chip counts take values in an additive subgroup of $\R$;
see the beginning of~\S\ref{subsec:related} for more discussion.

We refer the reader to~\cite{ratchip1} for additional chip-firing
motivations for studying $k$-skeletal objects; some of their
enumerative properties will become apparent in this article. It should
also be noted that our work from~\cite{ratchip1} builds on that
of~\cite{Backman-bij} and~\cite{caracciolo2012multiple}.

We adapt some definitions from~\cite{ratchip1} to our new chip-firing model.
For this, we must first recall the concept of borrow moves.
In either model discussed so far, let $T$ be any $t$-element 
set of non-sink vertices in $G$. The \emph{borrow move} $\beta_T$
is defined as the inverse of the cluster-fire move $\phi_{T}$,
thought of as a function acting on chip configurations.
In more detail, for the model in~\S\ref{subsec:int-chip-model},
$\beta_T$ acts on $\ch$ by increasing $\ch(j)$ by $c+(n-t)m$
for all $j\in T$ and decreasing $\ch(i)$ by $tm$ for all $i\in\compl{T}$.
For the model in~\S\ref{subsec:rat-chip-model1},
$\beta_T$ acts on $\ch$ by increasing $\ch(j)$ by $c+\lfloor (n-t)a/b\rfloor$
for all $j\in T$ and decreasing $\ch(i)$ by $\lfloor ta/b\rfloor$
for all $i\in\compl{T}$. A borrow move is \emph{legal} for a 
configuration $\ch\geq 0$ if $\beta_T(\ch)\geq 0$.

Say $G$ has $n$ non-sink vertices (so $n=b$ in the rational model of
\S\ref{sec:rat-chip-model}).
For $0\leq k<n$, we say a nonnegative chip configuration $\ch$ is
\emph{$k$-skeletal} if $\ch$ is $k$-stable and, for any legal borrow
move $\beta_T$ with $T\neq\varnothing$, $\beta_T(\ch)$ is not
$k$-stable. Since a legal borrow move $\beta_T$ with $|T| \leq k+1$
could be immediately followed by the legal $k$-firing $\beta_T^{-1}$,
it suffices to check legal borrow moves $\beta_T$ with $|T|>k+1$.
The $(n-1)$-skeletal configurations are the same as the
superstable configurations.  At the other extreme, a $0$-skeletal
configuration $\ch$ is a special kind of stable configuration: no
individual vertex can fire in $\ch$; but after any nonempty set of
vertices legally borrows in $\ch$, the new configuration always has
some vertex that can fire by itself.

\begin{example}\label{ex:ex75b}
We revisit Example~\ref{ex:ex75a}, where $a=7$ and $b=5$.
The $2$-stable configuration $22500$ is not $2$-skeletal,
because $\beta_{\{1,2,4,5\}}(22500)=44022$ is $2$-stable.
The $2$-stable configuration $44022$ is $2$-skeletal.
To check this, note there is no $T$ of size $4$ such that
$\beta_T(44022)\geq 0$. Taking $T=[5]$, we find
$\beta_T(44022)=55133$, which is not $2$-stable since 
$\phi_{\{1,2\}}$ is now legal. 
Thus, $44022$ is the unique
$2$-skeletal configuration that is a $2$-stabilization of $00355$.
\end{example}

In~\S\ref{sec:compare-skel} and~\S\ref{sec:group},
we prove the following facts about $k$-skeletal configurations.

\begin{theorem}\label{thm:kskel-chip}
Assume the chip-firing model of~\S\ref{subsec:rat-chip-model1}.
Fix $k$ with $0\leq k<b$.
\begin{enumerate}
\item For any configuration $\ch$, there exists a unique
$k$-skeletal chip configuration $\ch'$ reachable from $\ch$ via a finite
sequence of (not necessarily legal) cluster-firing and borrowing moves.
If $\ch\geq 0$, then we can convert $\ch$ to $\ch'$ using all legal moves.\label{thm:kskel-chip:unique}
\item There is a bijection from superstable configurations
 $\ch$ to $0$-skeletal configurations $\ch'$ given by
 $\ch'(i)=a-1-\lfloor a/b\rfloor-\ch(i)$ for $i\in [b]$.\label{thm:kskel-chip:duality}
\item There are $a^{b-1}$ $k$-skeletal chip configurations.\label{thm:kskel-chip:ab-1}
\item When $S_b$ acts on chip configurations, there are
 $\frac{1}{a+b}\binom{a+b}{a,b}$ orbits of $k$-skeletal chip configurations.\label{thm:kskel-chip:cat}
\end{enumerate}
\end{theorem}

By generalizing from the superstable case to the $k$-skeletal case, 
this theorem provides even stronger connections to the enumeration 
theory for rational-slope lattice paths, which we discuss next.

\subsection{\texorpdfstring{$k$}{k}-Skeletal Rational-Slope Lattice Paths}
\label{subsec:kskel-path-intro}

Section~\ref{sec:skel-paths} presents our combinatorial theory of
$k$-skeletal rational-slope lattice paths. We develop that theory
with no reference to chip-firing, although its definitions are
motivated by the quantized rational chip-firing model.
We summarize the main results here.

Fix relatively prime positive integers $a,b$.
We view lattice paths from $(0,0)$ to $(a,b)$ as sequences
of $b$ unit-length north steps and $a$ unit-length east steps.
Let $\mcP(\N^b\E^a)$ be the set of such paths.
We call $P\in\mcP(\N^b\E^a)$ an \emph{$(a,b)$-Dyck path} if 
all north steps of $P$ start on or above the line $bx=ay$. 
A classical cyclic-shifting argument (reviewed in~\S\ref{subsec:lev-cyc}) 
proves that the number of $(a,b)$-Dyck paths is the 
\emph{rational Catalan number} $\Cat_{a,b}=\frac{1}{a+b}\binom{a+b}{a,b}$.

For $Q\in\mcP(\N^b\E^a)$ and any lattice point $v\in\Z^2$ on $Q$, 
define $Q_v$ (the \emph{cyclic shift of $Q$ starting at $v$}) as follows. 
The steps of the path $Q_v$ (starting at the origin) 
consist of the steps of the path $Q$ leading from $v$ to $(a,b)$,
followed by the steps of the path $Q$ leading from $(0,0)$ to $v$.
For an integer $k$ with $0\leq k<b$, we call $Q\in\mcP(\N^b\E^a)$ a 
\emph{$k$-skeletal path for parameters $(a,b$)}
if and only if these two conditions hold:
\begin{itemize}
\item[(R1)] The last $k+1$ north steps of $Q$ start on or above $bx=ay$.
\item[(R2)] For each lattice point $v$ on $Q$ with $v$ strictly above $bx=ay$,
 (R1)~is false for $Q_v$.
\end{itemize}

\begin{example}\label{ex:skel-vs-dyck}
In the case $k=b-1$, Condition~(R1) means that $Q$ is an $(a,b)$-Dyck path.
For such a $Q$, and any $v$ on $Q$ strictly above $bx=ay$, the point
$(0,0)$ on $Q$ gets shifted to a point on $Q_v$ strictly below $bx=ay$.
The first north step of $Q_v$ after that point must start below $bx=ay$,
so (R1)~is false for $Q_v$. Thus $Q$ satisfies~(R2). In summary,
$(b-1)$-skeletal paths for $(a,b)$ are the same thing as $(a,b)$-Dyck paths.
\end{example}

We convert a path $Q\in\mcP(\N^b\E^a)$ to a \emph{labeled path}
as follows. Attach labels $1,2,\ldots,b$ (used once each)
to the $b$ unit-length north steps in $Q$ such that the labels of
consecutive north steps increase reading from bottom to top.
When cyclically shifting a labeled path, labels shift along with
the north steps they are attached to.

Labeled paths naturally encode chip configurations 
$\ch:[b]\rightarrow\Z$ where $0\leq \ch(i)\leq a$ for all $i\in [b]$.
(All $k$-stable configurations satisfy this condition,
 as noted in Remark~\ref{rem:basic-bound}.)
We make the labeled path encoding $\ch$, written $\lpath(\ch)$, as follows.
For $i,j\in [b]$, say $i$ is \emph{poorer} than $j$
(and $j$ is \emph{richer} than $i$) to mean $\ch(i)<\ch(j)$,
or $\ch(i)=\ch(j)$ and $i<j$. Let $v_1,v_2,\ldots,v_b$ be the 
vertices in $[b]$ listed in order from poorest to richest, and
let $x_1,x_2,\ldots,x_b$ be the corresponding chip counts
($x_k=\ch(v_k)$ for all $k$). For $k=1,2,\ldots,b$, draw a north step
in the Cartesian plane from $(x_k,k-1)$ to $(x_k,k)$ with label $v_k$.
There is a unique way to link together these north steps by adding east steps
to produce a labeled lattice path from $(0,0)$ to $(a,b)$.
This object is $\lpath(\ch)$, while the same path with labels erased
is denoted $\upath(\ch)$. When $S_b$ acts on configurations of a complete
graph $G$ by permuting the non-sink vertices, the orbit of $\ch$ corresponds
to the unlabeled path $\upath(\ch)$. 
We say a labeled path is \emph{$k$-skeletal} to mean the underlying
unlabeled path is $k$-skeletal.

We define $\lpath(\ch)$ and $\upath(\ch)$ similarly for any
nonnegative chip configuration $\ch$. The only catch is that
the last north step of the path may arrive at the line $y=b$ to the right
of the point $(a,b)$. In that case, we do not draw any east steps on this line.
Such paths are not $k$-skeletal for any $k$.

\begin{example}
Let $a=4$, $b=3$, and $\ch=202$, so $\ch(1)=\ch(3)=2$ and $\ch(2)=0$.
Figure~\ref{fig:lpath} shows $P = \lpath(\ch)$. Both $\ch$ and
$\lpath(\ch)$ are $0$-skeletal but not $1$-skeletal or
$2$-skeletal. Note that $P_v$ \emph{is} both $1$- and $2$-skeletal.
\begin{figure}[ht]
  {\scalebox{0.4}{\includegraphics{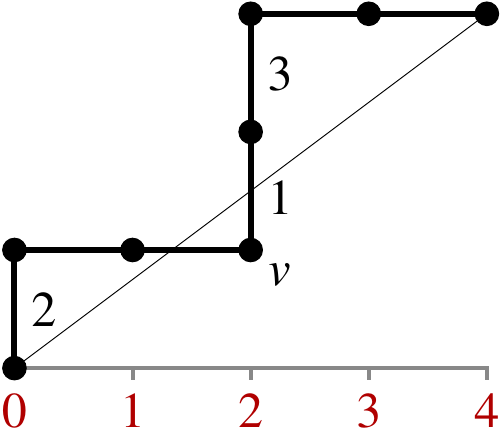}}}
  \caption{Labeled path $P=\lpath(\ch)$ for $\ch=202$.}
  \label{fig:lpath}
\end{figure}
\end{example}

Here is our main result for $k$-skeletal lattice paths.

\begin{theorem}\label{thm:main-rat}
Fix coprime integers $a,b>0$ and $k$ with $0\leq k<b$.
\begin{enumerate}
\item For each $P\in\mcP(\N^{\,b}\E^{\,a})$, there is exactly one $k$-skeletal 
path $Q$ such that $Q=P_v$ for some lattice point $v$ on $P$.
\label{thm:main-rat:uniq}

\item There are $\Cat_{a,b}$ $k$-skeletal paths for parameters $(a,b)$.
\label{thm:main-rat:cat}

\item For each labeled path $Q$ from $(0,0)$ to $(a,b)$ ending in an east
 step, there is exactly one $k$-skeletal cyclic shift of $Q$.
\label{thm:main-rat:shift}

\item There are $a^{b-1}$ labeled $k$-skeletal paths for parameters $(a,b)$.
\label{thm:main-rat:ab-1}

\end{enumerate}
\end{theorem}

We prove this theorem in~\S\ref{sec:skel-paths}.
Section~\ref{sec:compare-skel} translates these results into chip-firing
language and uses them to prove the facts stated earlier for $k$-skeletal
chip configurations in Theorem~\ref{thm:kskel-chip}.  
The main result of~\S\ref{sec:compare-skel} is
Theorem~\ref{thm:rat-skel}, which says that a nonnegative chip configuration
$\ch$ is $k$-skeletal if and only if $\lpath(\ch)$ is $k$-skeletal.

\subsection{Group Structure on Chip Configurations and Labeled Paths}
\label{subsec:group-intro}

We can give an algebraic interpretation of some of our results on
$k$-skeletal chip configurations for the quantized rational
chip-firing model.  Identify a chip configuration $\ch:[b]\rightarrow\Z$ 
with the vector of chip counts $\vv=(\ch(1),\ch(2),\ldots,\ch(b))$ in $\Z^b$.
Introduce an equivalence relation $\approx$ on $\Z^b$ by saying
that $\vv\approx\ww$ if we can convert $\vv$ to $\ww$ by
a finite sequence of (not necessarily legal) cluster-firing moves $\phi_S$
and borrow moves $\beta_T$ with no restriction on $|S|$ or $|T|$.
Let $\mcG=\Z^b/{\approx}$ be the set of equivalence classes for this
equivalence relation. In \S~\ref{sec:group}, 
we prove the following results.

\begin{theorem}\label{thm:group}
Fix coprime positive integers $a,b$.
\begin{enumerate}
\item Addition of integer vectors induces a group structure on $\mcG$
 that is isomorphic to $\Z_a^{b-1}$ (the direct product of $b-1$ copies
 of the integers modulo $a$).\label{thm:group:zab-1}
\item Every equivalence class in $\mcG$ contains nonnegative
 representatives.\label{thm:group:nonneg}
\item For any nonnegative $\ch,\ch'$ in the same equivalence class
 of $\mcG$, we can transform $\ch$ to $\ch'$ by a finite sequence of 
 cluster-fire and borrow moves that are all legal.\label{thm:group:seq}
\item For each $k$ with $0\leq k<b$, each equivalence class
 contains exactly one $k$-skeletal chip configuration.\label{thm:group:kskel}
\end{enumerate}
\end{theorem}

We can also interpret elements of $\mcG$ as equivalence classes
of labeled paths in $\mcP(\N^{\,b}\E^{\,a})$ ending in
an east step, relative to the equivalence relation on such paths
induced by cyclically shifting the steps of the path. Each
equivalence class contains exactly one labeled $(a,b)$-Dyck path,
which encodes the unique superstable chip configuration in that class.

\subsection{Related Prior Work}
\label{subsec:related}

In this section, we place the results of this article in the context of
the extensive work previously done by many researchers on lattice
paths, chip-firing games, and their connections. We refer the reader
to~\cite{corry2018divisors,klivans} for general introductions to this topic.
The chip-firing game, first introduced by~\cite{biggs,
  bjorner-lovasz-shor}, has a long history with connections to 
many different areas of mathematics including
statistical physics~\cite{bak, Dhar}, arithmetic
geometry~\cite{Lorenzini89,Lorenzini91,raynaud1970specialisation},
poset theory~\cite{mosesian1972strongly, pretzel1986reorienting, 
propp2002lattice}, 
lattice theory~\cite{amini2010riemann,bacher1997lattice,biggs},
tropical geometry~\cite{baker-norine,mikhalkin-zharkov},
commutative algebra~\cite{rossin, perkinson, postnikov-shapiro},  
and combinatorial representation 
theory~\cite{MR3484736,cori2016hall,Sandpile2024,Dukes-LeBorgne}.

First, we compare the quantized chip-firing model and $k$-skeletal
objects studied in this paper to the analogous concepts from our
recent paper~\cite{ratchip1}. Although the definitions of $k$-skeletal
chip configurations agree in these two settings, 
these definitions rely on different firing rules. 
In the chip-firing model of~\cite{ratchip1}, 
we select an additive subgroup $\mcG$ of $\mathbb{R}$ and
positive parameters $m,c\in\mcG$ representing
the capacities of non-sink edges and sink edges, respectively.
The chip counts at each vertex take values in $\mcG$
and need not be integers. When a vertex fires, it sends $c$ chips
to the sink and $m$ chips along every other outgoing edge, with
no rounding.

In both models, it is possible to define $k$-skeletal objects
and to represent chip configurations as certain labeled paths in the plane.
But the north steps of the paths used in~\cite{ratchip1} have $x$-coordinates
in $\mcG$ and need not be lattice paths.  
When $\mcG$ is not a discrete subgroup of $\R$,
there may be infinitely many $k$-skeletal chip configurations
and $k$-skeletal paths.  Although we can
recover the combinatorics of Fuss--Catalan lattice paths and parking functions
by taking $\mcG=\Z$, $c=1$, and $m\in\Z_{>0}$,
the model in~\cite{ratchip1} does not give a correspondence between superstable
chip configurations and general rational-slope parking functions.

The new quantized chip-firing model in this article overcomes this
limitation. The new model has nicer integrality and finiteness properties
compared to the general case of the other model. The Fuss--Catalan objects
can still be recovered as a special case of the new model too, since removing
the final east step from an $(mn+1,n)$-Dyck path maps such paths bijectively
to lattice paths in the triangle with vertices $(0,0)$, $(0,n)$, and $(mn,n)$.
Moreover, in the special case $a=mn+1$, $b=n$, and $c=1$, the firing rules
of the quantized model in~\S\ref{subsec:rat-chip-model1} reduce to the firing 
rules of the integral model in~\S\ref{subsec:int-chip-model}, 
as is readily checked.

Second, we discuss several generalizations of parking functions and
how these generalizations relate to chip-firing.  A \emph{parking
function of order $n$} (see~\cite{yan,Konheim,riordan}) is an
$n$-tuple of nonnegative integers $\alpha =
(\alpha_1,\alpha_2,\ldots,\alpha_n)$ such that when the parts are placed in
increasing order, say $\alpha_{i_1}\leq \alpha_{i_2}\leq \cdots \leq
\alpha_{i_n}$, we have $\alpha_{i_j} \leq j-1$ for all $j$ between $1$
and $n$.  Parking functions of order $n$ are in bijection with trees
on $n+1$ labeled vertices~\cite{kreweras}. It is a classical result
(see~\cite{Konheim}) that the number of parking functions of order $n$
is $(n+1)^{n-1}$. In the context of labeled trees, this counting
result is due to Cayley~\cite{cayley}.
Parking functions are central to the field of algebraic combinatorics. To
mention just one example of where they appear, the Frobenius series for the
$S_n$-module of diagonal coinvariants can be expressed as a weighted sum 
indexed by parking functions~\cite{carlsson,hhlru}.

Interpreting $\alpha$ as a chip configuration and replacing $\alpha$
by $\lpath(\alpha)$, we obtain the well-known identification of
parking functions with labeled lattice paths lying weakly above
$y=x$. By erasing the north-step labels (i.e., considering
$S_n$-orbits), we obtain the set of Dyck paths of order $n$. These are
enumerated by the Catalan numbers (see Stanley's compilations of
interpretations~\cite[Ex. 6.19]{ECII} and~\cite{addendum}).  Thus we
may expect that anywhere the Catalan numbers arise, parking functions
potentially do as well.

There are too many generalizations of parking functions to cover here,
but see~\cite{defective,zara,cchjr} for a few examples. For
the purposes of this article we are interested in two broad families:
vector parking functions and graphical parking functions.

Vector parking functions (also known as generalized parking functions
or $\bx$-parking functions) were introduced by Pitman and
Stanley~\cite{pitman}. Fix a vector $\bx = (x_1,x_2,\ldots,x_n)$ of
nonnegative integers. An \emph{$\bx$-parking function} is an $n$-tuple
$\alpha=(\alpha_1,\alpha_2,\ldots,\alpha_n)\in\Z^n_{\geq 0}$ 
such that for all $j$ between $1$ and $n$, the $j$th smallest entry of $\alpha$
is at most $x_1+x_2+\cdots +x_j-1$. Within this class of
vector parking functions, the classical parking functions correspond to
$\bx = (1,1,\ldots,1)$.  Results on vector parking functions 
may be found in~\cite{YanGPF,YanGen,KungYan}. 
The rational parking functions for slope $b/a$, 
appearing in~\cite{armstrong-ratcat,GMV} among other places,
are the special case of vector parking functions where $n=b$ and $\bx$
is chosen so that $\lpath(\alpha)$ lies in the triangle with vertices 
$(0,0)$, $(0,b)$, and $(a,b)$.

To define graphical parking functions,
we start with a loopless multigraph $G$ on the vertex set
$V=\{0,1,\ldots,n\}$. To explain the role of $G$,
we rephrase the condition on $\alpha$ given in our description of classical
parking functions: for any nonempty subset $U \subseteq
\{\alpha_1,\alpha_2,\ldots,\alpha_n\}$ of size $\ell$, 
there exists $\alpha_i\in U$ with $\alpha_i \leq n-\ell$. Write $\deg_U(i) =
\sum_{j\in V\setminus U} \omega(i,j)$ where $\omega(i,j)$ denotes
the number of edges between vertex $i$ and vertex $j$ in graph $G$. A
\emph{$G$-parking function} is an $n$-tuple $\beta =
(\beta_1,\beta_2,\ldots,\beta_n)$ such that
\begin{equation}\label{eq:Gpark}
  \text{for any nonempty subset $U\subseteq [n]$,
    there exists $i\in U$ such that $\beta_i \leq \deg_U(i)-1$.}
\end{equation}
The classical parking functions correspond to the case of
$G$ being the complete graph with $\omega(i,j)=1$ for all $i\neq j$.

Gaydarov and Hopkins~\cite{gaydarov2016parking} 
investigate the conditions under which
$G$-parking functions can be realized as $\bx$-parking
functions. Because of the inherent symmetry in the definition of
vector parking functions, $G$ must carry a corresponding symmetry.
Even assuming this property, it turns out the possibilities are quite
limited. Theorem~2.5 of~\cite{gaydarov2016parking} proves that
there is overlap in only three cases: 1)~$G$ is an $a$-tree; 
2)~$G$ is an $a$-cycle; or 3)~for some $c,m$, $\omega(0,i) = c$ 
for all $i\neq 0$ and $\omega(i,j)=m$ for all distinct and
nonzero $i,j$. In this third case, $G$ is a complete graph 
on $\{0,1,2,\ldots,n\}$ where sink edges have weight $c$
and non-sink edges have weight $m$; and the $G$-parking
functions are precisely the $\bx$-parking functions
for $\bx=(c,m,m,\ldots,m)$. This situation corresponds to the integral
chip-firing model from~\S\ref{subsec:int-chip-model}.

We introduce a well-behaved, quantized chip-firing model in
~\S\ref{subsec:rat-chip-model1} for which the resulting superstable
configurations are canonically in bijection with the rational parking
functions. This may be viewed as a generalization or extension of the
correspondence studied by Gaydarov and Hopkins. In particular,
consider the quantized chip-firing model on the complete graph as
introduced in~\S\ref{subsec:rat-chip-model1}. We interpret a $b$-tuple
$\beta = (\beta_1,\beta_2,\ldots,\beta_b)$ of nonnegative integers as
encoding the number of chips on the non-sink vertices. Such a
$b$-tuple $\beta$ encodes a superstable chip configuration if and only
if
\begin{equation}\label{eq:Gpark-quant}
  \text{for any nonempty subset $U\subseteq [b]$,
    there exists $i\in U$ such that $\beta_i \leq ((b-|U|)a/b + 1) - 1$.}
\end{equation}
Here, $(b-|U|)a/b + 1$ is the analogue of $\deg_U(i)$ on a
complete graph with edge weights as in \S\ref{subsec:rat-chip-model1}.
Note that the floor function is not necessary at this point and is
only needed when we start investigating the dynamics of quantized
rational chip-firing.

\section{General Quantized Rational Chip-Firing Model}
\label{sec:rat-chip-model}

\subsection{Definition of the Model}
\label{subsec:def-model}

This section defines the general version of our quantized rational
chip-firing model that was introduced in~\S\ref{subsec:rat-chip-model1}.
The model uses positive integer parameters $a,b,c$.
Let $G$ be a graph with edge set $E$ and vertex set $V=\{0\}\cup [n]$
for some positive integer $n$. Vertex $0$ is the \emph{sink}
and all other vertices are \emph{non-sink vertices}.
We assume $G$ is simple and undirected 
(meaning each $e\in E$ is a two-element subset of $V$)
and there is an edge from $0$ to $i$ for all $i\in [n]$.

For $i\in [n]$, let $\neigh(i)$ denote the set of non-sink
vertices adjacent to $i$, and write $\deg(i) = |\neigh(i)|$.
Informally, each edge from the sink to another vertex has capacity $c$,
while each edge between two non-sink vertices has capacity $a/b$.
As in the introduction, a \emph{chip configuration} is 
a function $\ch:[n]\rightarrow\Z$, and $\ch\geq 0$ means
$\ch(i)\geq 0$ for all $i\in [n]$.

We first define the \emph{firing move} $\phi_i$ for a non-sink vertex $i$.
By definition, $\phi_i(\ch)$
is the configuration obtained from $\ch$ by decreasing $\ch(i)$ by
$c+\lfloor \deg(i)a/b\rfloor$ and increasing $\ch(j)$ by 
$\lfloor a/b\rfloor$ for all $j\in\neigh(i)$. 
We say $i$ \emph{can (legally) fire} in configuration $\ch\geq 0$
if $\phi_i(\ch)\geq 0$, which means $\ch(i)\geq c+\lfloor \deg(i)a/b \rfloor$. 

Next we define the \emph{cluster-firing move} $\phi_S$ for
any subset $S$ of $[n]$. For $X\subseteq [n]$ and $i\in [n]$, 
write $\compl{X}=[n]\setminus X$, $\neigh_{X}(i) = \neigh(i)\cap X$, and 
$\deg_{X}(i) = |\neigh_{X}(i)|$.  By definition, $\phi_S(\ch)$
is the configuration obtained from $\ch$ by decreasing $\ch(i)$ by
$c+\lfloor \deg_{\compl{S}}(i)a/b\rfloor$ for each $i\in S$ 
and increasing $\ch(j)$ by $\lfloor
\deg_{S}(j)a/b\rfloor$ for all $j\in \compl{S}$. 
We say $S$ \emph{can (legally) fire} in configuration $\ch\geq 0$
if $\phi_S(\ch)\geq 0$, which means 
$\ch(i)\geq c+\lfloor \deg_{\compl{S}}(i)a/b\rfloor$ for each $i\in S$.
As before, $\phi_{\varnothing}$ is the identity map.
The \emph{borrow move} $\beta_S$ is $\phi_S^{-1}$.
This borrow move is \emph{legal} on $\ch\geq 0$ if $\beta_S(\ch)\geq 0$.

\begin{remark}
 In classical models of chip-firing, firing a vertex or a set of
 vertices conserves the total number of chips, once the sink is
 taken into account. In quantized firing, this is no longer the case
 as $\lfloor xa/b\rfloor \neq x\lfloor a/b\rfloor$ in general. 
 For example, take $a=2$, $b=3$, $c=1$, and $G$ to be the complete
 graph on $\{0,1,2,3\}$.  When a single vertex fires, it loses
 two chips even though the sink only gets one chip and the other
 two non-sink vertices gain none.
 Similarly, in the quantized rational chip-firing model, 
 firing several vertices sequentially is \emph{not} always the same as
 cluster-firing them, even when all the moves are legal.
\end{remark}

Let $0\leq k\leq n-1$.
A \emph{$k$-firing move} is a move $\phi_S$ where $0<|S|\leq k+1$.
A configuration $\ch\geq 0$ is \emph{$k$-stable} if no $k$-firing move 
is legal for $\ch$. If some finite sequence of legal $k$-firing moves
converts $\ch\geq 0$ to a $k$-stable $\ch'$, then we call $\ch'$
a \emph{$k$-stabilization} of $\ch$. We prove shortly that $k$-stabilizations
exist for all $k$ and that the $0$-stabilization of $\ch$ is unique.
The next example shows that $k$-stabilizations are not unique, in general,
even in the superstable case ($k=n-1$).

\begin{example}\label{ex:stab}
 Use parameters $a=2$, $b=5$, $c=1$, and the graph $G$ shown in
 Figure~\ref{fig:graph}.  Take $X=\{1,2\}$, $Y=\{2,3,4\}$ and $\ch=112100$. 
 First we compute $\phi_X(\ch)$ in full detail.
 When the cluster-fire move $\phi_X$ acts on $\ch$,
 vertex $1$ loses $1+\lfloor 2\cdot \frac{2}{5}\rfloor = 1$
 chip, vertex $2$ loses $1+\lfloor 1\cdot \frac{2}{5}\rfloor = 1$ chip,
 vertex $3$ gains $\lfloor 1\cdot \frac{2}{5}\rfloor = 0$ chips,
 and vertex $5$ gains $\lfloor 2\cdot \frac{2}{5}\rfloor = 0$ chips.
 Vertices $4$ and $6$ are unaffected by this firing move.
 In summary, $\phi_X(\ch)=002100$.
 By similar calculations, we find $\phi_{\{3,4\}}(002100)=000000$,
 which is certainly superstable.  However, $\phi_Y(\ch)=100010$
 and $\phi_1(100010)=000010$ is a different superstable configuration
 reachable from $\ch$.  On the one hand, this gives an example illustrating the
 non-uniqueness of superstabilizations. 
On the other hand, one can also check that the classical Diamond
Property~\cite[Thm. 2.2.2 (1)]{klivans} (also called ``local
confluence'') is not valid in this setting: There do not exist sets
$Z$ and $Z'$ that can be legally fired such that $\phi_{Z}\circ\phi_X(\ch) =
\phi_{Z'}\circ\phi_Y(\ch)$. (The above computations show that the
most natural candidates, $Z=Y\setminus X$ and $Z'=X\setminus Y$, don't
lead to a common configuration.)
\end{example}
  \begin{figure}[ht]
    \centering
    \includegraphics[width=0.3\linewidth]{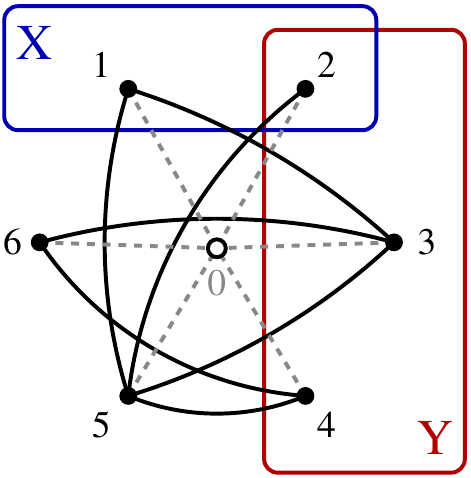}
    \caption{The graph from Example~\ref{ex:stab}. The black edges
      between non-sink vertices have capacity $2/5$ each. The gray,
      dotted edges incident with the sink vertex $0$ have capacity
      $1$ each.}
    \label{fig:graph}
  \end{figure}

\subsection{Existence of \texorpdfstring{$k$}{k}-Stabilizations}
\label{subsec:exist-k-stab}

We find a $k$-stabilization of a chip configuration $\ch\geq 0$
by repeatedly performing legal $k$-firing moves until a $k$-stable
configuration is reached. To settle the existence of $k$-stabilizations,
the only issue is proving that such a sequence of moves always terminates
in finitely many steps. For this purpose, we define the
\emph{total chip count} of any configuration $\ch$ to be
$\degch{\ch}=\sum_{i=1}^{n}\ch(i)$. As always, we disregard
the chip count at the sink. The key fact is that a non-identity
cluster-fire move always decreases the total chip count.

\begin{prop}\label{prop:fin}
  Given a chip configuration $\ch$ on $G$ and $\varnothing
  \subsetneq S \subseteq [n]$, $\degch{\phi_S(\ch)} < \degch{\ch}$.
\end{prop}
\begin{proof}
  Let $S = \{v_1,v_2,\ldots,v_\ell\}$ and $\compl{S} =
  \{w_1,w_2,\ldots,w_{n-\ell}\}$. Write $d_i = \deg_{\compl{S}}(v_i)$ 
  for $1\leq i\leq \ell$ and $e_j = \deg_S(w_j)$ for $1\leq j\leq n-\ell$. 
  When $S$ fires, each $v_i\in S$ loses $c+\lfloor d_ia/b \rfloor$ 
  chips, and each $w_j\in \compl{S}$ gains $\lfloor
  e_ja/b\rfloor$ chips. For any real $x$,
  write $\{x\}$ for the fractional part of $x$, 
  so $x = \lfloor x\rfloor + \{x\}$. Summing over all non-sink
  vertices, we find that the net change in the total chip count
  going from $\ch$ to $\phi_S(\ch)$ is
  \begin{align*}
    \sum_{j=1}^{n-\ell} \left\lfloor e_j\frac{a}{b}\right\rfloor 
- c\ell-\sum_{i=1}^{\ell} \left\lfloor d_i \frac{a}{b} \right\rfloor 
&= 
    -c\ell+\sum_{j=1}^{n-\ell} \left(e_j\frac{a}{b} 
- \left\{e_j\frac{a}{b}\right\}\right) 
- \sum_{i=1}^{\ell} \left(d_i \frac{a}{b}
- \left\{d_i \frac{a}{b} \right\}\right)\\
    &=-c\ell+\frac{a}{b}\left(\sum_{j=1}^{n-\ell} e_j
    - \sum_{i=1}^{\ell} d_i \right)
    + \sum_{i=1}^{\ell} \left\{d_i \frac{a}{b}\right\}
    - \sum_{j=1}^{n-\ell}\left\{e_j\frac{a}{b}\right\}.
\end{align*}
Now $\sum_{j=1}^{n-\ell} e_j=\sum_{i=1}^{\ell} d_i$ since both
sums count the edges joining some vertex in $S$ to some vertex in $\compl{S}$.
Next, since each $d_i$ is an integer, the fractional part of
$d_i\frac{a}{b}$ is at most $\frac{b-1}{b}$ for each $i$. 
Finally, by discarding the subtracted sum of $\{e_ja/b\}$, we obtain 
\[ \degch{\phi_S(\ch)}-\degch{\ch} \leq -c\ell+\ell(b-1)/b
 =\frac{\ell}{b}(-cb+b-1)<0, \]
since $c\geq 1$.
\end{proof}

\begin{prop}\label{prop:exist-kstab}
For each chip configuration $\ch\geq 0$ and each $k$ between 
$0$ and $n-1$, at least one $k$-stabilization of $\ch$ exists.
\end{prop}
\begin{proof}
The total chip count $\degch{\ch}$ is a nonnegative integer.
As we perform any sequence of legal $k$-firing moves on $\ch$,
the resulting configurations are all nonnegative and
the total chip count strictly decreases with each move.
Thus the sequence must terminate after finitely many moves,
and the final configuration is a $k$-stabilization of $\ch$.
\end{proof}

\begin{remark}
  We must place \emph{some} condition on the graph and weights in
  order to ensure there is always a $k$-stabilization. We have chosen
  to require the sink to be adjacent to every vertex and for all edges
  incident to the sink to have weight at least $1$. Suppose this is
  not the case: take a $3$-cycle on vertices $u$, $v$, $w$, with a
  sink attached as a leaf to $w$. Set $a/b=1/2$. If we cluster-fire
  $\{u,v\}$, vertex $w$ gains a chip and $u$ and $v$ are left
  unchanged. Hence no nonnegative chip configuration on this graph is
  $1$-stable. To prevent such scenarios, we have required that the
  sink be a drain on any vertex that gets fired.
\end{remark}

\subsection{Uniqueness of \texorpdfstring{$0$}{0}-Stabilizations}
\label{subsec:uniq-0-stab}

The uniqueness of $0$-stabilizations relies on the following
basic fact, which may be readily checked.

\begin{prop}[Diamond Property for $k=0$]\label{prop:diamond-prop}
Suppose $\ch\geq 0$ is a chip configuration and $i,j\in [n]$
are distinct vertices where $\phi_i$ and $\phi_j$ are both legal firing
moves for $\ch$. Then $\phi_j$ is legal for $\phi_i(\ch)$,
$\phi_i$ is legal for $\phi_j(\ch)$, and
$\phi_j(\phi_i(\ch)) = \phi_i(\phi_j(\ch))$.
\end{prop}

We already saw in Example~\ref{ex:stab} that the analogous property
for cluster-fires need not hold in this quantized rational chip-firing model.

\begin{theorem}\label{thm:unique0}
 For any graph $G$ and configuration $\ch\geq 0$,
 $\ch$ has exactly one $0$-stabilization.
\end{theorem}
\begin{proof}
Assume, to get a contradiction, that some nonnegative configuration
$\ch$ on $G$ has more than one $0$-stabilization. Choose such a $\ch$
where $\degch{\ch}$ is minimized. There must exist two sequences of
vertices, say $v_1,v_2,\ldots,v_s$ and $w_1,w_2,\ldots,w_t$
such that: $\phi_{v_1},\ldots,\phi_{v_s}$ is a sequence of legal
firings sending $\ch$ to a $0$-stable $\ch_1$;
$\phi_{w_1},\ldots,\phi_{w_t}$ is a sequence of legal firings
sending $\ch$ to a $0$-stable $\ch_2$; and $\ch_1\neq\ch_2$.
We must have $v_1\neq w_1$. Otherwise, $\ch'=\phi_{v_1}(\ch)=\phi_{w_1}(\ch)$
would be a nonnegative configuration having distinct
$0$-stabilizations $\ch_1$ and $\ch_2$ and where $\degch{\ch'}<\degch{\ch}$,
contradicting the choice of $\ch$.

By Proposition~\ref{prop:diamond-prop}, the configuration $\ch^* =
\phi_{w_1}(\phi_{v_1}(\ch)) = \phi_{v_1}(\phi_{w_1}(\ch))$ can be
obtained by a sequence of two legal firings from $\ch$. Choose
vertices $z_1,\ldots,z_u$ such that $\phi_{z_1},\ldots,\phi_{z_u}$ is
a sequence of legal firings sending $\ch^*$ to a $0$-stable $\ch_3$.
Now, $\ch_1$ and $\ch_3$ are both $0$-stabilizations of $\phi_{v_1}(\ch)$,
which has smaller total chip count than $\ch$; so $\ch_1=\ch_3$.
Also, $\ch_2$ and $\ch_3$ are both $0$-stabilizations of $\phi_{w_1}(\ch)$,
which has smaller total chip count than $\ch$; so $\ch_2=\ch_3$.
We have reached the contradiction $\ch_1=\ch_2$.
\end{proof}

\begin{remark}\label{rem:int-kstab-uniq}
In the model of~\S\ref{subsec:int-chip-model}, the following version
of the diamond property is true and easy to check.
Given $\ch\geq 0$ and distinct nonempty subsets $A$ and $B$ such that
$\phi_A$ and $\phi_B$ are both legal for $\ch$, 
$\phi_{B\setminus A}$ is legal for $\phi_A(\ch)$,
$\phi_{A\setminus B}$ is legal for $\phi_B(\ch)$,
and $\phi_{B\setminus A}(\phi_A(\ch))=\phi_{A\setminus B}(\phi_B(\ch))$.
The technique used to prove Theorem~\ref{thm:unique0} adapts to
prove the uniqueness of $k$-stabilizations (for any $k$) in 
this integral chip-firing model.
\end{remark}

\begin{remark}\label{rem:real}
Here is an even more general version of our quantized chip-firing model.
We specify nonnegative real weights $w_{ij}=w_{ji}$
for all distinct $i,j\in V=\{0\}\cup [n]$, subject to the condition
that $w_{i0}\geq 1$ for all $i\in [n]$. For any $S\subseteq [n]$
and configuration $\ch$, $\phi_S(\ch)$ is obtained from $\ch$
by decreasing $\ch(i)$ by $\lfloor w_{i0}+\sum_{j\in \compl{S}} w_{ij}\rfloor$
for each $i\in S$ and increasing $\ch(j)$ by
$\lfloor \sum_{i\in S} w_{ij}\rfloor$ for all $j\in\compl{S}$.
Proposition~\ref{prop:fin} and the other results of this section
readily extend to this setting.
\end{remark}

Although the general model is appealing, we need to impose some
homogeneity hypotheses to obtain strong enumerative results.
So, for the rest of this article, we restrict to the special case
of the quantized chip-firing model already discussed 
in~\S\ref{subsec:rat-chip-model1}, namely: 
$c=1$, $\gcd(a,b)=1$, and $G$ is the complete graph on
$\{0\}\cup [b]$. In this setting, the superstabilization of $\ch\geq 0$
is unique. We will deduce this as a special case of the theory
of $k$-skeletal chip configurations. To develop that theory,
we first study the analogous results for lattice paths.

\section{\texorpdfstring{$k$}{k}-Skeletal Lattice Paths}
\label{sec:skel-paths}

In this section, we develop a theory of $k$-skeletal lattice paths
(unlabeled and labeled) with no specific reference to chip-firing. The
next section will prove the connection to $k$-skeletal configurations.
We assume the definitions from~\S\ref{subsec:kskel-path-intro}, but otherwise
this section is self-contained. 
Our goal is to prove Theorem~\ref{thm:main-rat}.
Throughout, we fix coprime positive integers $a,b$.

\subsection{Levels and Cyclic-Shift Equivalence Classes}
\label{subsec:lev-cyc}

Define the \emph{level} of a lattice point $v=(x,y)\in\Z^2$
(relative to the coprime pair $(a,b)$ or the line $bx=ay$)
to be $\lvl(v)=ay-bx$. Points on the line $bx=ay$ have level $0$;
points above or left of this line have positive levels; and
points below or right of this line have negative levels.
Moving north from $(x,y)$ to $(x,y+1)$ increases the level by $a$,
while moving east from $(x,y)$ to $(x+1,y)$ decreases the level by $b$.
The following lemma follows routinely from the hypothesis
$\gcd(a,b)=1$ (cf.~\cite[Thm.~12.1]{loehr-comb}).

\begin{lemma}\label{lem:levels}
All lattice points in the rectangle $[0,a]\times [0,b]$ have
distinct levels, except $(0,0)$ and $(a,b)$ both have level 0.
\end{lemma}

In particular, the levels of the lattice points reached by any path
$Q\in\mcP(\N^b\E^a)$ are all distinct, except $(0,0)$ and $(a,b)$
both have level $0$.
Because of this, we can specify cyclic shifts of $Q$ by
writing $Q_{\lvl(v)}$ instead of $Q_v$. In other words, if $Q$ visits
a (necessarily unique) lattice point at level $\ell\neq 0$, then $Q_{\ell}$ 
is the cyclic shift of $Q$ obtained by moving this lattice point
to the origin. Also write $Q_0=Q$.

Next, define an equivalence relation $\sim$ on $\mcP(\N^b\E^a)$ 
by letting $P\sim Q$ mean that $Q$ is a cyclic shift of $P$.
Define a statistic ${\min:\mcP(\N^b\E^a)\rightarrow\Z_{\leq 0}}$ by letting
$\min(Q)$ be the minimum level of any lattice point visited by $Q$.
The following observations are readily checked. Note that part~(\ref{lem:rat-facts:equiv})
follows from part~(\ref{lem:rat-facts:min}) and Lemma~\ref{lem:levels}.

\begin{lemma}\label{lem:rat-facts}
Suppose $Q\in\mcP(\N^{\,b}\E^{\,a})$ and $v$ is a lattice point on $Q$ of level 
$\ell=\lvl(v)$.
\begin{enumerate}
 \item $Q$ is an $(a,b)$-Dyck path if and only if $\min(Q)=0$.
 \item $\min(Q_v)=\min(Q_{\ell})=\min(Q)-\ell$. \label{lem:rat-facts:min}
 \item The equivalence class of $Q$ has size $a+b$.\label{lem:rat-facts:equiv}
 \item Every equivalence class of $\sim$ contains exactly one $(a,b)$-Dyck path.
 \label{lem:rat-facts:d}
\end{enumerate}
\end{lemma}

Since the set $\mcP(\N^b\E^a)$ of size $\binom{a+b}{a,b}$ has been written
as the disjoint union of equivalence classes that each have size $a+b$,
there are $\Cat_{a,b}=\frac{1}{a+b}\binom{a+b}{a,b}$ such equivalence
classes. Part~(\ref{lem:rat-facts:d}) 
of Lemma~\ref{lem:rat-facts} implies that the
number of $(a,b)$-Dyck paths is $\Cat_{a,b}$.
Theorem~\ref{thm:main-rat}(\ref{thm:main-rat:uniq}) can now be rephrased 
as the following equivalent statement.

\begin{prop}\label{prop:main-rat}
For fixed $k$ with $0\leq k<b$, 
each equivalence class of $\sim$ contains exactly one $k$-skeletal path.
\end{prop}

We prove this proposition in the next subsection.
Once this is done, we get a bijection from the set of $k$-skeletal paths
to the set of $(a,b)$-Dyck paths by mapping a $k$-skeletal path $Q$
to the unique $(a,b)$-Dyck path $P$ such that $P\sim Q$,
namely $P=Q_{\min(Q)}$.  This gives a bijective proof 
of Theorem~\ref{thm:main-rat}(\ref{thm:main-rat:cat}).

\subsection{Proof of Proposition~\ref{prop:main-rat}}
\label{subsec:proof-main-rat}

Fix an equivalence class $T$ of $\sim$, and let $P$ be the unique
$(a,b)$-Dyck path belonging to this equivalence class.
Let $0=\ell_1<\ell_2<\cdots<\ell_{a+b}$ be the $a+b$ distinct
levels visited by $P$. We totally order $T$ by minimum level:
\[ T=\{P=P_{\ell_1}>_T P_{\ell_2}>_T\cdots>_T P_{\ell_{a+b}}\},
\quad\mbox{where $\min(P_{\ell})=-\ell$.} \]

\begin{lemma}
For any path $Q$ of the form $P_{\ell_i}$, 
\[ \{Q_w:\text{$w$ is on $Q$ and $\lvl(w)>0$}\} =\{R\in T: Q>_T R\}. \]
\end{lemma}
\begin{proof}
To prove one set inclusion, suppose $w$ is on $Q$ and $\lvl(w)>0$.
Then $Q_w\sim Q\sim P$ belongs to $T$ and $\min(Q_w)=\min(Q)-\lvl(w)<\min(Q)$,
so $Q >_T Q_w$. To prove the other set inclusion, 
consider $R\in T$ with $Q>_T R$, say $R=P_{\ell_j}$ with $j>i$.
If we let $w$ be the vertex on $Q$ with level $\ell_j-\ell_i>0$, then $R=Q_w$.
Such a vertex does exist, 
since $P$ has a vertex at level $\ell_j$ which gets moved
to level $\ell_j-\ell_i$ in $Q$ and moved to level $0$ in $R$.
\end{proof}

For any $Q\in\mcP(\N^b\E^a)$, define $\pos(Q)=\pos_{a,b}(Q)$ to be the maximum
number of north steps (scanning backwards from the end of $Q$ and not
skipping any north steps, but allowing intervening east steps)
that start at vertices with nonnegative level. We can now 
reformulate the definition of $k$-skeletal paths. A path
$Q\in\mcP(\N^b\E^a)$ is $k$-skeletal for the coprime pair $(a,b)$
if and only if these two conditions hold:
\begin{itemize}
\item[(R1$'$)] $\pos(Q)>k$.
\item[(R2$'$)] For all $R\sim Q$ with $Q>_T R$, $\pos(R)\leq k$.
\end{itemize}
In (R2$'$), $T$ is the equivalence class of $\sim$ containing $Q$.

Proposition~\ref{prop:main-rat} is now easy to prove.
Given $k$ with $0\leq k<b$ and an equivalence class $T$ of $\sim$,
the unique $k$-skeletal path in $T$ is the smallest $Q\in T$ (relative
to the total ordering $>_T$) satisfying $\pos(Q)>k$. Such a $Q$ must 
exist, since the $(a,b)$-Dyck path $P$ in $T$ has $\pos(P)=b>k$.

\begin{example}
Take $a=8$, $b=5$, and $P=\N\E\N\E\E\N\N\E\E\N\E\E\E$.
\begin{figure}[ht]
  {\scalebox{0.5}{\includegraphics{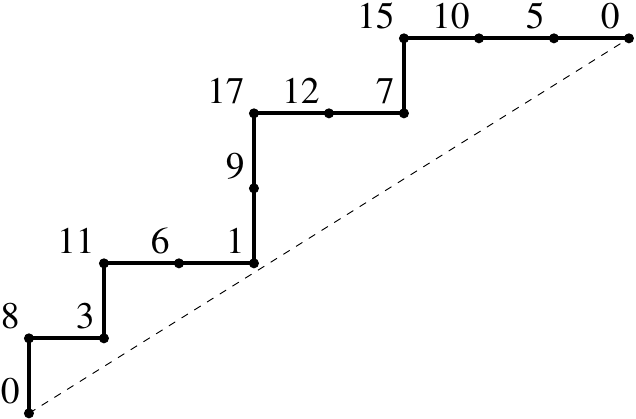}}}
  \caption{An $(8,5)$-Dyck path with levels of lattice points marked.}
  \label{fig:find-skel}
\end{figure}
Figure~\ref{fig:find-skel} shows $P$ with the level of each 
lattice point $v$ on $P$ marked. 
We list each path $P_{\ell}$ in the equivalence class of $P$ (totally ordered 
by $>_T$) and the value of $\pos(P_{\ell})$.
\begin{center}
\begin{tabular}{lccccccccccccc}\toprule
\text{shift of $P$:} &
$P_0$ & $P_1$ & $P_3$ & $P_5$ & $P_6$ & $P_7$ & $P_8$ & $P_9$ & $P_{10}$ & $P_{11}$
& $P_{12}$ & $P_{15}$ & $P_{17}$ \\\midrule
\text{$\pos(P_{\ell})$:} & 
$5$ & 1 & 0 & 2 & 0 & 1 & 0 & 0 & 0 & 0 & 0 & 0 & 0\\\bottomrule
\end{tabular}
\end{center}

The unique $k$-skeletal path in this equivalence class is 
$P=P_0$ for $k\in\{2,3,4\}$, $P_5$ for $k=1$, and $P_7$ for $k=0$.
\end{example}

\subsection{Labeled \texorpdfstring{$(a,b)$}{(a,b)}-Paths}
\label{subsec:label-rat}

Given a path $Q\in\mcP(\N^b\E^a)$, a \emph{run} of north steps
in $Q$ is a maximal string of consecutive north steps in $Q$.
Let $\run(Q)$ be the \emph{run multiset of $Q$}, which is the multiset
of lengths of all such runs of north steps in $Q$.
Recall that we convert $Q$ to a \emph{labeled path}
by attaching a permutation of $[b]$ to the $b$ north steps of $Q$
where the labels in each run of north steps increase from bottom to top.
Let $\mcLP(\N^b\E^a)_E$ be the set of labeled paths from $(0,0)$ to
$(a,b)$ that end in an east step. 
All labeled $(a,b)$-Dyck paths --- in fact, all $k$-skeletal labeled paths ---
belong to this set, by condition~(R1).

View the integers modulo $a$ as the set 
of remainders $\Z_a=\{0,1,2,\ldots,a-1\}$.
The set of chip configurations $\ch:[b]\rightarrow\Z$ 
where $0\leq \ch(i)<a$ for all $i\in [b]$ can be identified
with the product set $\Z_a^b$. The encoding of functions by labeled paths 
(\S\ref{subsec:kskel-path-intro})
restricts to a bijection from $\Z_a^b$ to $\mcLP(\N^b\E^a)_E$.
We introduce an equivalence relation $\equiv$ on each of these sets, 
as follows. For $\ch,\ch'\in\Z_a^b$, let $\ch\equiv\ch'$ mean
that for some $j\in\Z_a$, for all $x\in [b]$, $\ch'(x)=(\ch(x)+j)\bmod a$. 

When $\ch'$ is related to $\ch$ by this formula, we obtain
$\lpath(\ch')$ from $\lpath(\ch)$ by cyclically shifting $\lpath(\ch)$
(viewed as a sequence of east steps and labeled north steps) so that
the $(j+1)$th east step from the end of $\lpath(\ch)$ gets shifted to
become the last east step of $\lpath(\ch')$. This shifting does not
break any run of consecutive north steps, so the requirement that
labels increase reading up each run is preserved.  Rephrasing in terms
of labeled paths $Q,Q'\in\mcLP(\N^b\E^a)_E$, we can say $Q\equiv Q'$
if and only if $Q'$ is a cyclic shift of $Q$.  In this case,
$\run(Q)=\run(Q')$, so we sometimes speak of $\equiv$ as the
\emph{run-preserving cyclic shift relation} on labeled paths. 
(We use $\sim$ for paths related by arbitrary cyclic shifts and
$\equiv$ for ones related by run-preserving cyclic shifts. When the
paths in question end in east steps there is no ambiguity as any shift
that preserves membership in $\mcLP(\N^b\E^a)_E$ must be
run-preserving.)
We now restate and prove parts (\ref{thm:main-rat:shift}) and
(\ref{thm:main-rat:ab-1}) of Theorem~\ref{thm:main-rat}.
 
\begin{theorem}\label{thm:labeled-rat}
Fix $k$ in the range $0\leq k<b$.
\begin{enumerate}
  \item For every $Q\in\mcLP(\N^{\,b}E^{\,a})_E$, there is exactly one
 $k$-skeletal labeled path $Q'$ with $Q\equiv Q'$.
  \item The number of $k$-skeletal labeled paths
for the coprime pair $(a,b)$ is $a^{b-1}$. In particular,
the number of labeled $(a,b)$-Dyck paths is $a^{b-1}$.
\end{enumerate}
\end{theorem}
\begin{proof}
Given $Q\in\mcLP(\N^b\E^a)_E$, let $P$ be $Q$ with the labels erased.
By Proposition~\ref{prop:main-rat},
there is a unique cyclic shift $P'$ of $P$ (with $P'\sim P$)
that is $k$-skeletal. Since $P'$ must end in an east step, the same
cyclic shift of the labeled path $Q$ preserves runs and converts
$Q$ to a $k$-skeletal labeled path $Q'$ with $Q'\equiv Q$. 
$Q'$ must be unique since $P'$ is. This proves~(a).

Part~(a) says that for each $k$, every equivalence class of
$\equiv$ contains exactly one $k$-skeletal labeled path.
So it suffices to show there are exactly $a^{b-1}$ such equivalence classes.
The set of labeled paths under consideration has size $|\Z_a^b|=a^b$.
Each labeled path is equivalent (under run-preserving cyclic shifts)
to exactly $a$ labeled paths, since any of the $a$ east steps may be
shifted to become the final east step. Thus, each equivalence class
has size $a$, and there are $a^b/a=a^{b-1}$ equivalence classes.
\end{proof}

This completes the proof of Theorem~\ref{thm:main-rat}.

\section{Comparison of \texorpdfstring{$k$}{k}-Skeletal Paths 
and Configurations}
\label{sec:compare-skel}

This section continues to study quantized rational chip-firing
assuming the following setup: $a,b$ are coprime positive integers; $G$
is a complete graph on $\{0\}\cup [b]$; and sink edges have capacity
$c=1$.  Our main result is that the correspondence between chip
configurations $\ch$ and labeled paths $\lpath(\ch)$ preserves the
property of being $k$-skeletal.

\begin{theorem}\label{thm:rat-skel}
Fix $k$ between $0$ and $b-1$.
A nonnegative chip configuration $\ch$ is $k$-skeletal if and only if
$\lpath(\ch)$ is $k$-skeletal. 
\end{theorem}

The counts in parts~(\ref{thm:kskel-chip:ab-1})
and~(\ref{thm:kskel-chip:cat}) of Theorem~\ref{thm:kskel-chip} follow
by combining Theorem~\ref{thm:rat-skel} with
parts~(\ref{thm:main-rat:cat}) and~(\ref{thm:main-rat:ab-1}) of
Theorem~\ref{thm:main-rat}.

We prove Theorem~\ref{thm:rat-skel} in the next four subsections.  In
\S\ref{subsec:analyze-k-stable} we recast the $k$-stability of a chip
configuration in terms of conditions on a lattice path (specifically,
Condition (R1)). In the subsequent two subsections, we significantly
reduce the checking required by Condition (R2) of $k$-skeletal
paths. Section~\ref{subsec:prove-thm-ratskel} ties together these results
to finish the proof of Theorem~\ref{thm:rat-skel}.

Recall from~\S\ref{subsec:kskel-path-intro} the concept of one vertex
being richer (or poorer) than another in a configuration $\ch$.  In
terms of $\lpath(\ch)$, vertex $i$ is poorer than vertex $j$ precisely
when (north-edge) label $i$ appears in a lower row than (north-edge)
label $j$ in the labeled path. In particular, ``the $s$ poorest
vertices of $\ch$'' refers to the labels in the $s$ lowest rows of the
diagram; ``the $s$ richest vertices of $\ch$'' refers to the labels in
the top $s$ rows of the diagram; ``the $s$th richest vertex of $\ch$''
refers to the label in row $s$ from the top; and so on.

We also write $\Delta$ for the triangle in $\R^2$ (including its
interior) with vertices $(0,0)$, $(0,b)$, and $(a,b)$. This is the
light-yellow triangle illustrated in Figure~\ref{fig:tri2}. The line
$x=(a/b)y$ is called the \emph{diagonal} of $\Delta$.

\subsection{Analysis of \texorpdfstring{$k$}{k}-Stability}
\label{subsec:analyze-k-stable}

\begin{lemma}\label{lem:test-stable}
For $0\leq k<b$, a nonnegative chip configuration $\ch$ is $k$-stable
if and only if $\lpath(\ch)$ satisfies Condition~(R1):
namely, the last $k+1$ north steps of $\lpath(\ch)$ start on or above 
the diagonal of $\Delta$.
\end{lemma}
\begin{proof}
Fix $s$ between $1$ and $b$. Let the north step of $\lpath(\ch)$
in the $s$th row from the top start at the lattice point $(x_0,b-s)$.
Note $x_0$ is the chip count at the $s$th richest vertex of $\ch$.
The following conditions are equivalent.
\begin{itemize}
\item There exists an $s$-element subset of $[b]$ that
 can legally fire in configuration $\ch$.
\item The set of $s$ richest vertices of $\ch$ can legally fire.
\item The $s$ richest vertices of $\ch$ each have at least 
 $1+\lfloor (b-s)a/b\rfloor$ chips.
\item The $s$th richest vertex of $\ch$ has at least 
 $1+\lfloor (b-s)a/b\rfloor$ chips.
\item $x_0\geq 1+\lfloor (b-s)a/b\rfloor$.
\end{itemize}
Negating these conditions, we get another list of equivalent conditions:
\begin{itemize}
\item No $s$-element subset of $[b]$ can legally fire in configuration $\ch$.
\item $x_0\leq \lfloor (b-s)a/b\rfloor$.
\item $x_0\leq (b-s)a/b$ (equivalent to the previous item
since $x_0$ must be an integer).
\item $(x_0,b-s)$ is weakly left of the diagonal $x=(a/b)y$.
\item The $s$th north step from the top in $\lpath(\ch)$ starts in $\Delta$.
\end{itemize}
Taking $s=1,2,\ldots,k+1$, we now see that $\ch$ is $k$-stable if
and only if all the north steps in the top $k+1$ rows of $\lpath(\ch)$
start in $\Delta$, which is Condition~(R1).
\end{proof}

Because of this lemma, we call a lattice path $P\in\mcP(\N^b\E^a)$
\emph{$k$-stable} if and only if $P$ satisfies Condition~(R1).
Hence, rewriting conditions (R1) and (R2)
from~\S\ref{subsec:kskel-path-intro}, a lattice path $Q$ (with or
without labels) is $k$-skeletal if and only if 1)~$Q$ is $k$-stable;
and 2)~for each lattice point $v$ on $Q$ strictly above the diagonal
of $\Delta$, the cyclic shift $Q_v$ is not $k$-stable.

\subsection{Min-Level Points}
\label{subsec:min-points}

Define the \emph{min-level points} for parameters $(a,b)$ 
to be the lattice points
\begin{equation}\label{eq:min-lvl}
 v_s=\left(\lfloor sa/b \rfloor,s \right)\quad
 \mbox{for $s=1,2,\ldots,b-1$;}\quad v_b=(a-1,b).
\end{equation}
These are the rightmost lattice points in each row strictly left of
the diagonal $x=(a/b)y$. Because $\gcd(a,b)=1$, the levels of the 
points $v_1,v_2,\ldots,v_b$ must be positive and distinct 
(cf.~Lemma~\ref{lem:levels}). Each such level is at most $b$,
since taking one east step from $v_s$ leads weakly east of the diagonal.
Thus the levels of $v_1,v_2,\ldots,v_b$ are $1,2,\ldots,b$ in some order,
and every other lattice point of $\Delta$ (excluding those on the diagonal)
has level greater than $b$.

The next result reduces the number of lattice points $v$ we must check
in Condition~(R2) when deciding if a lattice path $Q$ is $k$-skeletal
for parameters $(a,b)$ (cf.~\cite[Prop.~3.17]{ratchip1}).

\begin{lemma}\label{lem:test-skel1}
Let $Q$ be a $k$-stable lattice path in $\mcP(\N^{\,b}\E^{\,a})$.
Path $Q$ is not $k$-skeletal if and only if 
there exists a min-level point $v_s$ on $Q$, reached by an east step of $Q$,
such that $Q_{v_s}$ is $k$-stable.
\end{lemma}
\begin{proof}
Negating the definition of ``$k$-skeletal,'' we see that a $k$-stable path $Q$
is not $k$-skeletal if and only if there exists a point $v$ on $Q$,
with $\lvl(v)>0$, such that $Q_v$ is $k$-stable.
It suffices to prove that when such $v$ exists, there also exists
a point $v_s$ with the properties stated in the lemma.
The $k$-stability of $Q_v$ means that the top $k+1$ north steps
in the shifted path $Q_v$ start at weakly positive levels. Equivalently,
the $k+1$ north steps immediately preceding $v$ on $Q$ (as we scan
southwest along $Q$ starting at $v$, wrapping from $(0,0)$ to $(a,b)$
if needed) all start at levels weakly exceeding $\lvl(v)$.

To find $v_s$, scan northeast along $Q$ from $v$, taking zero or more steps,
until reaching the first min-level point $v_s$ on $Q$. Such a point must
exist, since $v$ is strictly left of the diagonal, $Q$ returns to the
diagonal at $(a,b)$, and $Q$ can only reach the diagonal from the left
side by passing through one of the min-level points. Consider the
$k+1$ north steps preceding $v_s$ on $Q$. Each such north step either
precedes or follows $v$. Such a north step preceding $v$ must also be
one of the $k+1$ north steps mentioned earlier, so its starting level
$\ell$ satisfies $\ell\geq \lvl(v)\geq \lvl(v_s)$. Such a north step following
$v$ also has level $\ell\geq\lvl(v_s)$ since $v_s$ is the earliest min-level
point following $v$ on $Q$, and $Q$ cannot go outside $\Delta$ between
$v$ and $v_s$. These remarks show that $Q_{v_s}$ is $k$-stable.
Finally, $Q$ must reach $v_s$ via an east step, since otherwise the first
north step preceding $v_s$ would start at a level below $\lvl(v_s)$.
\end{proof}

\subsection{Analysis of Borrow Moves}
\label{subsec:analyze-borrow}

To continue, we examine the effect of a borrow move on the labeled path
of a chip configuration $\ch\geq 0$. Fix $s$ with $1\leq s\leq b$. 
Define $E(s)=1+\lfloor (b-s)a/b\rfloor$ and $W(s)=\lfloor sa/b\rfloor$.
Call the line $x=W(s)$ the \emph{critical line} (for borrow moves 
involving $s$ vertices). Let $S$ be an $s$-element subset of $[b]$.
We can execute the borrow move $\beta_S(\ch)$ by acting on the labeled
path of $\ch$ as follows. For each north step labeled by some $i\in S$,
move that north step east by $E(s)$ units.
For each north step labeled by some $j\not\in S$,
move that north step west by $W(s)$ units.  Here and below,
labels always move with the north steps they are attached to.
To finish, we get the labeled path of $\beta_S(\ch)$ by sorting the rows
(based on the new locations of the north steps) so that vertices
appear richest to poorest scanning from the top of the diagram.

For the borrow move $\beta_S$ to be legal on $\ch$, the north steps that 
move west must all start weakly right of the critical line $x=W(s)$. 
Equivalently, $\beta_S$ is legal on $\ch$ if and only if all north steps 
starting strictly left of the critical line have labels in $S$. The next lemma
describes a special situation where doing a borrow move on $\ch$
corresponds to a run-preserving cyclic shift of the associated labeled path.

\begin{lemma}\label{lem:special-borrow}
Suppose a configuration $\ch\geq 0$ is $0$-stable and
$Q=\lpath(\ch)$ arrives at the min-level point $v_s$ by an east step.
Let $S$ be the set of $s$ poorest vertices in $\ch$. Then $S$ can legally
borrow in configuration $\ch$, and $Q_{v_s}=\lpath(\beta_S(\ch))$.
\end{lemma}
\begin{proof}
Let the north steps of $Q$, from bottom to top, have $x$-coordinates
$x_1\leq x_2\leq \cdots\leq x_b$, and define $x_{b+1}=a$. 
The assumption that $Q$ arrives at
$v_s=(W(s),s)$ by an east step means that $x_s<W(s)\leq x_{s+1}$.
Since $S$ is the set of labels on the $s$ lowest north steps,
the paragraph preceding the lemma shows that $\beta_S$ is legal for $\ch$.
Executing this borrow move on the path diagram, the $s$ lowest north
steps move to new $x$-coordinates $x_1+E(s)\leq x_2+E(s)\leq\cdots
\leq x_s+E(s)$, while the $b-s$ highest north steps move to new
$x$-coordinates $x_{s+1}-W(s)\leq\cdots\leq x_b-W(s)$. 
Because $\ch$ is $0$-stable, $Q$ must end
in an east step and $x_b<a$ (Lemma~\ref{lem:test-stable}). 

In the special case $s=b$, we have $v_s=(a-1,b)$ and $E(s)=1$,
so the borrow move $\beta_{[b]}$ shifts all north steps of $Q$ right $1$ unit.
On the other hand, $Q_{v_s}$ is $Q$ with the final east step cyclically
shifted to the front. So the lemma holds in this case.

In the main case $1\leq s<b$, $sa/b$ is not an integer 
(otherwise $\gcd(a,b)=1$ forces $b$ to divide $s$). 
So $\lfloor -sa/b\rfloor =-\lfloor sa/b\rfloor-1$ and 
$E(s)+W(s)=a+1+\lfloor -sa/b\rfloor +\lfloor sa/b\rfloor=a$.
Therefore $x_b-W(s)<a-W(s)=E(s)\leq x_1+E(s)$. This means that
when we sort rows to finish executing $\beta_S$ on the path diagram,
the $b-s$ richest vertices in $\ch$ become the $b-s$ poorest vertices
in $\beta_S(\ch)$. More specifically, to get $\lpath(\beta_S(\ch))$,
we move each labeled north step in the top $b-s$ rows of $Q$ west $W(s)$
units and south $s$ units, and we move each labeled north step in the bottom
$s$ rows of $Q$ east $E(s)=a-W(s)$ units and north $b-s$ units. 
This has exactly the same effect as cyclically shifting all steps of $Q$
so that $v_s=(W(s),s)$ gets moved to the origin.
So $Q_{v_s}=\lpath(\beta_S(\ch))$, as needed.
\end{proof}

\begin{example}
Take $a=8$, $b=5$, and $\ch=13035$. The labeled path of $\ch$ 
appears on the left in Figure~\ref{fig:borrow-shift}. This
path arrives at the min-level point $v_2=(3,2)$ by an east step.
We compute $\beta_{\{1,3\}}(\ch)=\ch'=60502$ by adding
$E(2)=1+\lfloor 3(8/5)\rfloor=5$ to $\ch(1)$ and $\ch(3)$ as well as 
subtracting $W(2)=\lfloor 2(8/5)\rfloor=3$ from $\ch(2)$, $\ch(4)$,
and $\ch(5)$. The labeled path of $\ch'$ appears on the right in
the figure. As predicted by the lemma, $\lpath(\ch')$ is obtained
by cyclically shifting $\lpath(\ch)$ so that $v_2$ moves to the origin.
\begin{figure}[ht]
  {\scalebox{0.4}{\includegraphics{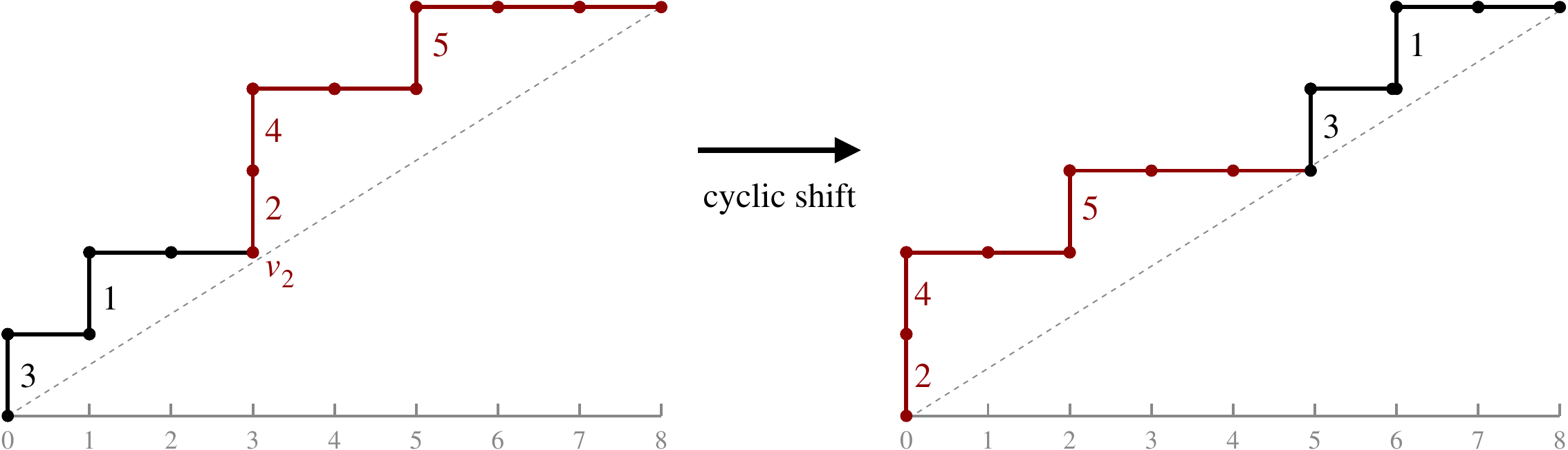}}}
  \caption{Example illustrating Lemma~\ref{lem:special-borrow}.}
  \label{fig:borrow-shift}
\end{figure}
\end{example}

The next lemma shows that, when testing whether a $k$-stable chip
configuration is $k$-skeletal, we need only consider borrow moves
involving the poorest $s$ vertices. In fact, only those values of $s$
satisfying the technical condition in the lemma need be tested.

\begin{lemma}\label{lem:test-skel}
For a $k$-stable chip configuration $\ch$ with labeled path $Q=\lpath(\ch)$,
$\ch$ is not $k$-skeletal if and only if for some $s\in\{1,2,\ldots,b\}$,
$Q$ arrives at the min-level point $v_s$ by an east step, 
$\beta_S$ is legal on $\ch$, and $\beta_S(\ch)$ is $k$-stable, 
where $S$ is the set of $s$ poorest vertices in $\ch$.
\end{lemma}
\begin{proof}
By definition, $\ch$ is not $k$-skeletal if and only if there exists a
nonempty subset $S$ of $[b]$ such that $\beta_S$ is legal on $\ch$ and
$\beta_S(\ch)$ is $k$-stable. It suffices to show that, when such an $S$
exists, it must have the additional properties stated in the lemma.
Say such an $S$ has size $s$. Assuming that $W(s)>0$ and $s<b$ for now,
$Q$ must arrive at the critical line $x=W(s)$ via some
east step ending at some point $(W(s),y_0)$. Consider the three possible cases.
\begin{itemize}
\item[(1)] If $y_0>s$, then $Q$ has more than $s$ north steps starting strictly
left of the critical line. At least one of these north steps must have a
label not in $S$, but that means $\beta_S$ is not legal on $\ch$
(by the paragraph before Lemma~\ref{lem:special-borrow}), 
contrary to choice of $S$.
\item[(2)] If $y_0<s$, then some vertex $v\in S$ must label a north step
starting weakly right of the critical line. After the borrow move at $S$,
the chip count at $v$ is at least $W(s)+E(s)=a$. But that means
$\beta_S(\ch)$ is not $0$-stable and hence not $k$-stable, contrary
to choice of $S$.
\item[(3)] If $y_0=s$, then $Q$ arrives at $v_s=(W(s),s)$ by an east step.
For $\beta_S$ to be legal on $\ch$, the $s$ north steps strictly
left of the critical line must have labels in $S$. But this forces
the $s$-element set $S$ to consist of the $s$ poorest vertices in $\ch$.
\end{itemize}
The situation $W(s)=0$ is handled similarly to the second 
case above. Here $s<b$ and $E(s)=a$, so that the borrow move $\beta_S$
leaves each vertex in $S$ with at least $a$ chips. This means that
$\beta_S(\ch)$ is not $k$-stable, contrary to choice of $S$. 
In the situation $s=b$, we must have $S=[b]$, so $\beta_S$ acts by 
adding $1$ to every vertex's chip count. Here $Q$ (being $k$-stable) must pass 
through $v_s=(a-1,b)$.  If $Q$ arrived at $v_s$ by a north step, 
then $\beta_S(\ch)$ would not be $k$-stable, contrary to choice of $S$. 
So $Q$ arrives at $v_s$ by an east step, and the
conclusion of the lemma is still satisfied.
\end{proof}

\begin{remark}\label{rem:Dhar}
  Lemma~\ref{lem:test-skel} is related to Dhar's Burning Algorithm
  (see~\cite[\S 2.6.7]{klivans} and ~\cite{Dhar}). Dhar's algorithm
  can be used to determine whether or not a chip configuration on an
  arbitrary graph is superstable, i.e., $(b-1)$-stable.
  In~\cite{Backman-bij}, a generalization of Dhar's algorithm was
  provided which applied to arbitrary hereditary chip-firing models.
  The situation of Lemma~\ref{lem:test-skel} is both more and less
  general. It is more general in that it can be used (along with
  Theorem~\ref{thm:rat-skel}) to check whether a chip configuration is
  $k$-skeletal for the rational quantized chip-firing model (not only
  for $k=b-1$, but for any $0\leq k\leq b-1$).  However, it is more
  restrictive in that it only applies to configurations on the
  complete graph. By taking advantage of this symmetry, we arrive at a
  simple test rather than a true multi-stage algorithm.
\end{remark}

\begin{example}\label{ex:ex75c}
We illustrate Lemma~\ref{lem:test-skel} by
revisiting Examples~\ref{ex:ex75a} and~\ref{ex:ex75b}, where $a=7$ and $b=5$.
The left picture in Figure~\ref{fig:not-skel} is $\lpath(22500)$.
This $2$-stable path arrives at the min-level point $v_4=(5,4)$ by an east step,
and the four poorest vertices are $4,5,1,2$. The right picture in the figure
is $\lpath(44022)$, where $44022=\beta_{\{1,2,4,5\}}(22500)$ is $2$-stable.
Thus, the condition in the lemma holds for the non-$2$-skeletal configuration
$22500$. On the other hand, consider $\lpath(44022)$, which arrives at min-level points
$v_1=(1,1)$, $v_3=(4,3)$, and $v_5=(6,5)$ via east steps.
We can check that $\beta_{\{3\}}(44022)$, $\beta_{\{3,4,5\}}(44022)$,
and $\beta_{[5]}(44022)$ are not $2$-stable. By Lemma~\ref{lem:test-skel},
these checks are enough to conclude that $44022$ is $2$-skeletal.
(In fact, since $|\{3\}|\leq 2+1$, $\phi_{\{3\}} = \beta_{\{3\}}^{-1}$
is automatically a legal $2$-firing move on $\beta_{\{3\}}(44022)$,
implying that $\beta_{\{3\}}(44022)$ is not $2$-stable. The same logic
applies to $\{3,4,5\}$, so only the $2$-stability of
$\beta_{[5]}(44022)=55133$ needs to be checked explicitly.)
\begin{figure}[ht]
  {\scalebox{0.4}{\includegraphics{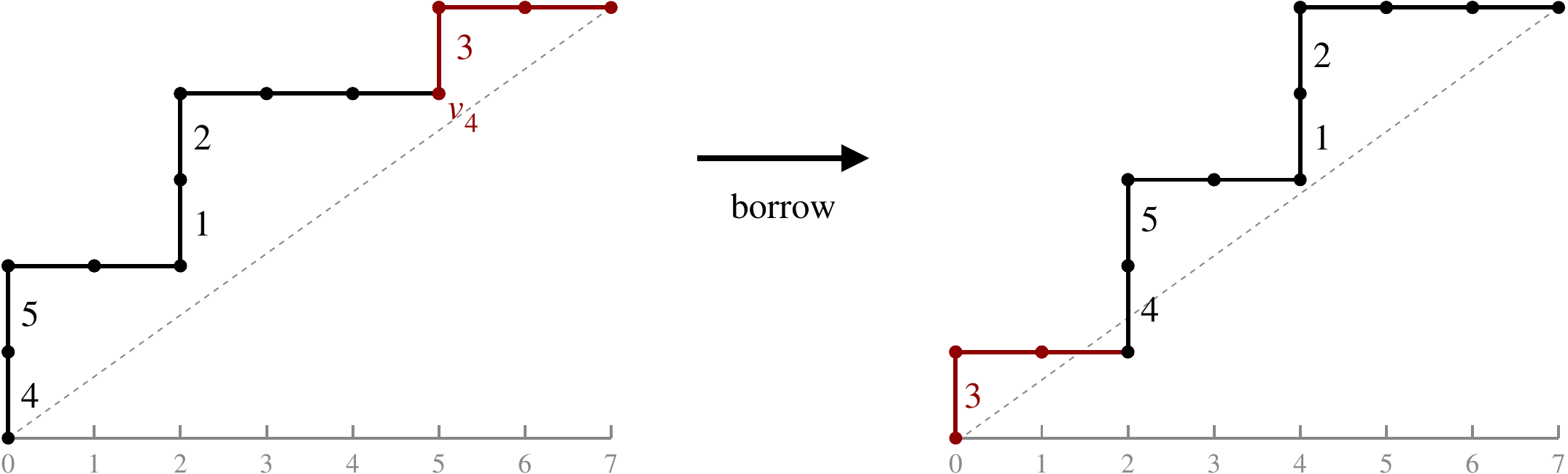}}}
  \caption{The labeled paths for $22500$ and $44022$.}
  \label{fig:not-skel}
\end{figure}
\end{example}

\subsection{Proof of Theorem~\ref{thm:rat-skel}}
\label{subsec:prove-thm-ratskel}

With the previous lemmas in hand, we can quickly prove
Theorem~\ref{thm:rat-skel}. Consider a configuration $\ch\geq 0$
and $Q=\lpath(\ch)$.  First, $\ch$ is $k$-stable if and only if
$Q$ is $k$-stable (Lemma~\ref{lem:test-stable}). For the
rest of the proof, look at the case where $\ch$ and $Q$ are $k$-stable.
Lemma~\ref{lem:test-skel} says that $\ch$ is not $k$-skeletal if and only if
there exist $s$ and $S$ (as described in that lemma)
where $\beta_S$ is legal on $\ch$ and $\beta_S(\ch)$ is $k$-stable.
Because the hypotheses of Lemma~\ref{lem:special-borrow} are met here,
we can replace ``$\beta_S(\ch)$ is $k$-stable'' by
``$Q_{v_s}$ is $k$-stable'' in the previous sentence.
This produces the condition in Lemma~\ref{lem:test-skel1},
which is equivalent to $Q$ not being $k$-skeletal. \qed

\subsection{Duality of \texorpdfstring{$0$}{0}-Skeletal and \texorpdfstring{$(b-1)$}{(b-1)}-Skeletal Objects}
\label{subsec:duality}

This section establishes a duality between $(b-1)$-skeletal
(superstable) configurations and $0$-skeletal configurations.  This
duality is a generalization of a classical duality between
superstable configurations and recurrent configurations which is related to
Riemann-Roch duality (see~\cite{asadichip,baker-norine,Cori-LeBorgne}
and also~\cite{manjunath2013monomials}).  Informally, the
$(b-1)$-skeletal configurations correspond to labeled $(a,b)$-Dyck
paths, while the $0$-skeletal configurations correspond to the same
paths turned upside down.

\begin{figure}[ht]
  {\scalebox{0.3}{\includegraphics{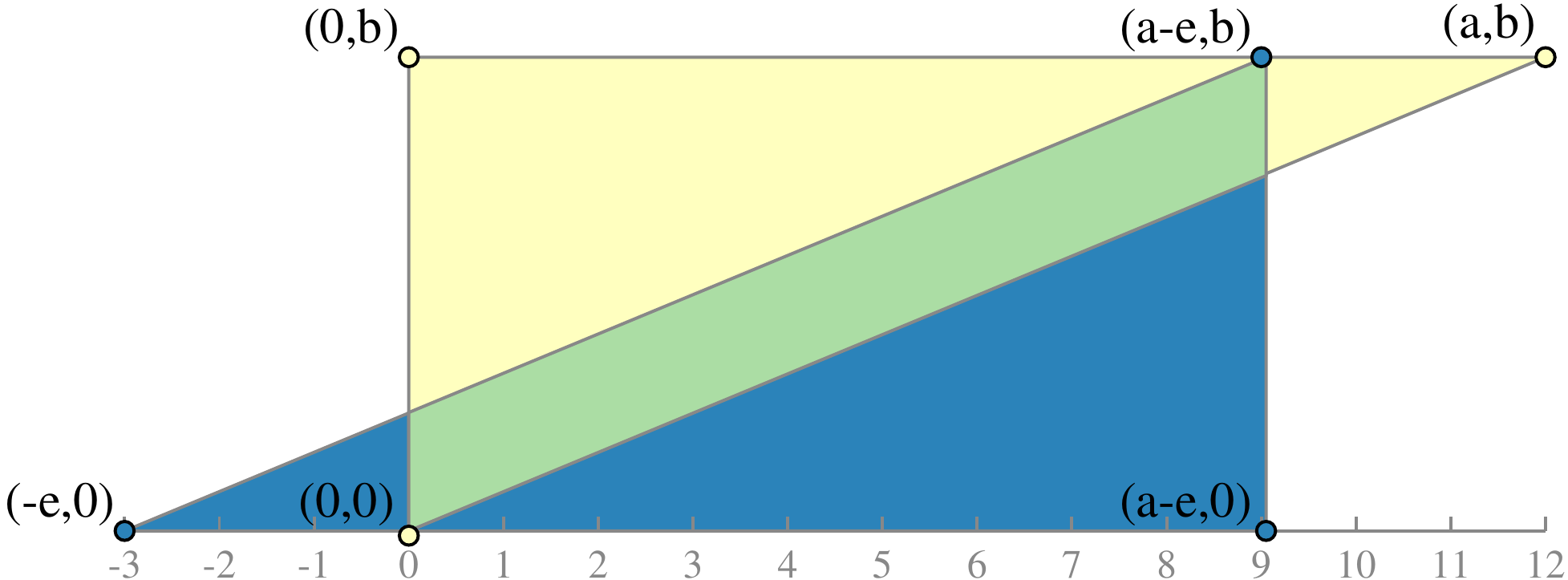}}}
  \caption{The triangle $\Delta$, in light yellow, and the inverted
    triangle $\Delta'$, in dark blue, where $a=12$, $b=5$, and
    $e=1+\lfloor 12/5\rfloor = 3$.}
  \label{fig:tri2}
\end{figure}

We set up some geometric notation for discussing this situation.  Fix
coprime positive integers $a,b$ with $b>1$. Recall $\Delta$ is the
triangle with vertices $(0,0)$, $(0,b)$, and $(a,b)$.  Define
$e=E(b-1)=1+\lfloor a/b\rfloor$, and recall $W(b-1)=a-e$ is the
$x$-coordinate of the min-level point $v_{b-1}$ at height $b-1$.  By
Lemma~\ref{lem:test-stable}, any $0$-stable path from $(0,0)$ to
$(a,b)$ must end in at least $e$ east steps. Next, define $\Delta'$ to
be the triangle with vertices $(-e,0)$, $(a-e,0)$, and $(a-e,b)$,
which can be viewed as $\Delta$ rotated 180 degrees.
Figure~\ref{fig:tri2} shows these two triangles in the case $a=12$,
$b=5$, and $e=3$. Given any $0$-stable path $Q$ from $(0,0)$ to
$(a,b)$, let $\mvl{Q}$ be the lattice path from $(-e,0)$ to $(a-e,b)$
obtained from $Q$ by moving the last $e$ east steps in $Q$ to become
the first $e$ east steps in $\mvl{Q}$. Note that 
$\run(Q)=\run(\mvl{Q})$, so the
operation sending $Q$ to $\mvl{Q}$ also makes sense for labeled paths.
Finally, call a path (with or without labels) from $(-e,0)$ to
$(a-e,b)$ that stays on or in $\Delta'$ an \emph{inverted $(a,b)$-Dyck
path}.

\begin{theorem}\label{thm:duality-paths}
Let $Q\in\mcP(\N^{\,b}\E^{\,a})$ be $0$-stable.
\begin{enumerate}
\item $Q$ is $(b-1)$-skeletal if and only if $Q$ is an $(a,b)$-Dyck path.
\item $Q$ is $0$-skeletal if and only if $\mvl{Q}$ is an inverted
 $(a,b)$-Dyck path.
\item There is a bijection from $(a,b)$-Dyck paths $Q$ to 
 inverted $(a,b)$-Dyck paths $Q^*$, where $Q^*$ is obtained by reversing
 the sequence of steps of $Q$ and starting at $(-e,0)$.
\item Analogues of (a), (b), and (c) hold for labeled paths.
\end{enumerate}
\end{theorem}
\begin{proof}
Part~(a) follows from Example~\ref{ex:skel-vs-dyck}.
To prove~(b), note first that the $0$-stable path $Q$ must end
in at least $e$ east steps (Lemma~\ref{lem:test-stable}), 
so $\mvl{Q}$ is well-defined. We use the criteria of Lemma~\ref{lem:test-skel1}
to show that $Q$ is not $0$-skeletal if and only if $\mvl{Q}$ does not
stay in $\Delta'$. 

On one hand, if $Q$ is not $0$-skeletal, pick
a min-level point $v_s$ on $Q$, reached by an east step of $Q$,
such that $Q_{v_s}$ is $0$-stable. Using Lemma~\ref{lem:test-stable} again,
the $0$-stability of the cyclic shift $Q_{v_s}$ means precisely that
$v_s$ is preceded on $Q$ by at least $e$ east steps. But $v_s$
is strictly left of the diagonal of $\Delta$ (the line $x=(a/b)y$),
so the point on $Q$ that is $e$ units west of $v_s$ must be 
strictly left of the diagonal of $\Delta'$ (which is the line $x=(a/b)y-e$).
The point in question is also on $\mvl{Q}$, so $\mvl{Q}$ 
does not stay in $\Delta'$.

Conversely, suppose $\mvl{Q}$ does not stay inside $\Delta'$. 
For $s=1,2,\ldots,b$, let $v_s'=v_s-(e,0)$ be the point $e$ units
west of the min-level point $v_s$. The last lattice point outside
$\Delta'$ visited by $\mvl{Q}$ must lie on $Q$ and be one of the points $v_s'$.
$Q$ must proceed from $v_s'$ by taking $e$ east steps to reach $v_s$,
since a north step taken fewer than $e$ east steps after $v_s'$
would take $Q$ to a point of higher level than $v_s'$. Such a point
would be outside $\Delta'$ and contradict the choice of $v_s'$
as the last such point on $\mvl{Q}$. Since $v_s$ is preceded on $Q$ by
at least $e$ east steps, $Q_{v_s}$ is $0$-stable, and $Q$ is not $0$-skeletal.

Part~(c) is geometrically evident (cf.~Figure~\ref{fig:tri2}).
Part~(d) follows since the maps sending $Q$ to $\mvl{Q}$ or $Q^*$ do not
break up any runs of north steps in $Q$. When we reverse a labeled
path $Q$ to get $Q^*$ (carrying labels along with their north steps), 
we must then sort the labels in each column to be increasing from bottom to top.
\end{proof}

There is an analogous duality between superstable chip 
configurations and $0$-skeletal chip configurations.

\begin{theorem}\label{thm:duality-configs}
Let $\ch$ be a $0$-stable chip configuration.
\begin{enumerate}
\item $\ch$ is $(b-1)$-skeletal (superstable)
 if and only if $\lpath(\ch)$ is a labeled $(a,b)$-Dyck path.
\item $\ch$ is $0$-skeletal if and only if $\mvl{\lpath(\ch)}$ 
 is a labeled inverted $(a,b)$-Dyck path.
\item There is a bijection from superstable configurations
 $\ch$ to $0$-skeletal configurations $\ch'$ given by
 $\ch'(i)=a-e-\ch(i)$ for $i\in [b]$.
\end{enumerate}
\end{theorem}
\begin{proof}
Parts~(a) and~(b) are immediate from Theorems~\ref{thm:rat-skel}
 and~\ref{thm:duality-paths}. For part~(c), it suffices to note that
 the path reversal operation from parts~(c) and~(d) of Theorem~\ref{thm:duality-paths}
 moves north step labels in column $x$ to column $a-e-x$, 
 for all $x$ in the range $0\leq x\leq a-e$. Thus a vertex with
 $x$ chips in $\ch$ becomes a vertex with $a-e-x$ chips in $\ch'$.
\end{proof}

Part~(c) of Theorem~\ref{thm:duality-configs} is simply a restatement
of part~(\ref{thm:kskel-chip:duality}) of Theorem~\ref{thm:kskel-chip},
hence the only part of Theorem~\ref{thm:kskel-chip} that remains to be
proved is part~(\ref{thm:kskel-chip:unique}), which we take care of in
\S\ref{sec:group}.

\section{Group Structure for the Quantized Rational Chip-Firing Model}
\label{sec:group}

We continue to assume the following setup: $a,b$ are coprime positive
integers; $G$ is a complete graph on $\{0\}\cup [b]$; and sink edges
have capacity $c=1$.  This section studies a natural group structure
for the quantized rational chip-firing model, which may be viewed as
an analogue of the classical critical group~\cite{biggs} (also known
as the sandpile group~\cite{Dhar}, chip-firing group~\cite{klivans},
or Jacobian~\cite{bacher1997lattice}; see~\cite{lorenzini}). It is
closely related to the Picard group. Our approach will be to first
introduce an abstract group $\Ga$ and then relate it to
chip-firing. Collectively, the lemmas below will prove
Theorem~\ref{thm:group} (see remarks following Lemma~\ref{lem:chK}).

Let $\Z_a$ be the additive group of integers modulo $a$.
Let $\ones=(1,1,\ldots,1)$ be the vector consisting of $b$ copies of $1$.
For $1\leq j\leq b$, let $\ee_j$ be the vector with a $1$ 
in position $j$ and $0$s elsewhere. Depending on context, 
we view $\ones$ and $\ee_j$ as vectors in $\Z^b$ or in $\Z_a^b$.
Let $K$ be the subgroup of $\Z^b$ generated by $a\ee_1,\ldots,a\ee_b,\ones$.
Let $H$ be the subgroup of $\Z_a^b$ generated by $\ones$.
Let $\Ga$ be the quotient group $\Z_a^b/H$, 
which can be identified with the isomorphic quotient group $\Z^b/K$.
Since $|\Z_a^b|=a^b$ and $|H|=a$, $\Ga$ is a finite commutative group
of size $a^{b-1}$.  The map $(v_1,v_2,\ldots,v_{b-1})\mapsto 
(v_1,\ldots,v_{b-1},0)+H$ is a group isomorphism $\Z_a^{b-1}\cong \Ga$,
as is readily checked. The key point is that each coset of $H$
contains a unique vector in $\Z_a^b$ with last component zero.

Given $\ch,\ch'\in\Z_a^b$, we have $\ch+H=\ch'+H$ (equality of 
cosets in $\Ga$) if and only if $\ch\equiv\ch'$ where $\equiv$
is the equivalence relation from~\S\ref{subsec:label-rat}.
As explained in that section, we can also view elements of 
$\Ga=\Z_a^b/H$ as equivalence classes of labeled paths in
$\mcLP(\N^b\E^a)_E$ under the run-preserving cyclic shift relation.
Theorem~\ref{thm:labeled-rat}(a) proved that, for each $k$ with
$0\leq k<b$, each such equivalence class contains exactly one 
$k$-skeletal representative.

Our main goal here is to prove that $\Ga$ can also be viewed as 
the set of equivalence classes of chip configurations for the
relation $\approx$ defined in~\S\ref{subsec:group-intro}.
Recall we identify a chip configuration $\ch:[b]\rightarrow\Z$ with the
vector of chip counts $(\ch(1),\ch(2),\ldots,\ch(b))$ in $\Z^b$,
and $\ch\approx\ch'$ means we can convert $\ch$ to $\ch'$ by
a finite sequence of (not necessarily legal) cluster-firing moves $\phi_S$
and borrow moves $\beta_T=\phi_T^{-1}$. 
The next lemmas prove the key fact that the equivalence relation $\approx$
coincides with congruence mod $K$.
Specifically, Lemmas~\ref{lem:eqrel1} and~\ref{lem:chK} prove part~(a)
of Theorem~\ref{thm:group}, which amounts to saying that the group
$\Z^b/{\approx}$ of firing-equivalent chip configuration classes under
pointwise addition is isomorphic to $\Z^b/K\cong\Ga\cong\Z_a^{b-1}$.

\begin{lemma}\label{lem:eqrel1}
For all $\ch,\ch'\in\Z^b$, if $\ch\approx\ch'$, then $\ch+K=\ch'+K$.
\end{lemma}
\begin{proof}
It suffices to prove this 
assuming $\ch'=\beta_S(\ch)$ for some $S\subseteq [b]$.
If $S=\varnothing$, then $\ch'=\ch$ and the result is clear.
If $S=[b]$, then $\ch'=\ch+\ones\in \ch+K$, as needed.
Now suppose $|S|=s$ where $0<s<b$. As shown in~\S\ref{subsec:analyze-borrow},
$\ch'=\beta_S(\ch)$ has entries $\ch'_i=\ch_i+a-W(s)$ for all $i\in S$
and $\ch'_j=\ch_j-W(s)$ for all $j\in\compl{S}$, 
where $W(s)=\lfloor sa/b\rfloor=a-E(s)$.
Thus, $\ch'=\ch-W(s)\ones+\sum_{i\in S} a\ee_i\in \ch+K$, as needed.
\end{proof}

Lemma~\ref{lem:eqrel1} shows that every equivalence class of $\approx$
is a subset of some coset of $K$.

\begin{lemma}\label{lem:get-nonneg}
For all $\ch\in\Z^b$, there exists $\ch'\geq 0$ with $\ch\approx\ch'$.
\end{lemma}
\begin{proof}
The borrow move $\beta_{[b]}$ (where all non-sink vertices borrow)
adds $1$ to the chip count at every vertex. Given $\ch\in\Z^b$
that may have negative entries, apply $\beta_{[b]}$ to $\ch$ finitely many 
times to reach some $\ch'\geq 0$ with $\ch\approx\ch'$.
\end{proof}

\begin{lemma}\label{lem:get-ss}
For all $\ch\in\Z^b$ with $\ch\geq 0$, there exists a superstable
 $\ch_{ss}\in\Z_a^b$ with $\ch\approx\ch_{ss}$. We can convert
 $\ch$ to $\ch_{ss}$ via a finite sequence of legal cluster-fire moves.
\end{lemma}
\begin{proof}
Apply Proposition~\ref{prop:exist-kstab} taking $k=b-1$.
\end{proof}

\begin{lemma}\label{lem:kK-uniq}
Fix $k$ with $0\leq k<b$.  Every coset of $K$ contains exactly one 
$k$-skeletal representative.
\end{lemma}
\begin{proof}
Suppose $\ch_1,\ch_2$ are $k$-skeletal and $\ch_1+K=\ch_2+K$.
Because $\ch_1,\ch_2\in\Z_a^b$, we can also say $\ch_1+H=\ch_2+H$,
meaning $\lpath(\ch_1)\equiv\lpath(\ch_2)$. By Theorem~\ref{thm:labeled-rat}(a),
$\lpath(\ch_1)=\lpath(\ch_2)$ and so $\ch_1=\ch_2$.
\end{proof}

\begin{prop}\label{prop:ss-uniq}
Every equivalence class of $\approx$ contains exactly one 
superstable representative. In particular,
every configuration $\ch\geq 0$ has a unique superstabilization.
\end{prop}
\begin{proof}
The existence assertions follow from Lemmas~\ref{lem:get-nonneg}
and~\ref{lem:get-ss}. For uniqueness,
suppose $\ch_1,\ch_2$ are both superstable and $\ch_1\approx\ch_2$.
By Lemma~\ref{lem:eqrel1}, $\ch_1+K=\ch_2+K$. 
So $\ch_1=\ch_2$ by Lemma~\ref{lem:kK-uniq}.
The second uniqueness statement follows since any two superstabilizations 
of $\ch$ are in the same equivalence class of $\approx$.
\end{proof}

\begin{lemma}\label{lem:chK}
For all $\ch,\ch'\in\Z^b$, if $\ch+K=\ch'+K$, then $\ch\approx\ch'$.
\end{lemma}
\begin{proof}
We know there is exactly one superstable
$\ch_{ss}\approx \ch$ and exactly one superstable $\ch_{ss}'\approx\ch'$.
By Lemma~\ref{lem:eqrel1}, $\ch_{ss}+K=\ch+K$ and $\ch_{ss}'+K=\ch'+K$.
Since $\ch_{ss}$ and $\ch_{ss}'$ are superstable representatives of
the same coset of $K$, they are equal (Lemma~\ref{lem:kK-uniq} 
with $k=b-1$). We deduce
$\ch\approx \ch_{ss}=\ch_{ss}'\approx\ch'$, so $\ch\approx\ch'$.
\end{proof}

The lemmas also justify Fact~(F2) from \S\ref{subsec:rat-chip-model1}
and Theorem~\ref{thm:kskel-chip}(\ref{thm:kskel-chip:unique}). It was
noted at the beginning of this section that Lemmas~\ref{lem:eqrel1}
and~\ref{lem:chK} prove part~(\ref{thm:group:zab-1}) of
Theorem~\ref{thm:group}. In addition, part~(\ref{thm:group:nonneg}) of
Theorem~\ref{thm:group} is the content of Lemma~\ref{lem:get-nonneg},
part~(\ref{thm:group:seq}) follows from Lemma~\ref{lem:chK} since all
moves used in the proof to go from $\ch$ to $\ch'$ are legal when
$\ch,\ch'\geq 0$, and part~(\ref{thm:group:kskel}) follows from
Lemma~\ref{lem:kK-uniq} in conjunction with Lemmas~\ref{lem:eqrel1}
and~\ref{lem:chK}.

Given a configuration $\ch$, we can compute the unique superstable $\ch_{ss}$
with $\ch+K=\ch_{ss}+K$ as follows. First, reduce
all entries in $\ch$ modulo $a$. Second, draw the labeled path
$Q\in\mcLP(\N^b\E^a)_E$ associated with the reduced configuration.
Third, find the unique vertex $v_0$ on $Q$ such that $\lvl(v_0)$
is minimized. The configuration $\ch_{ss}$ encoded by the cyclically shifted
path $Q_{v_0}$ is the superstable configuration we seek.
 
Similarly, for fixed $k$, we can proceed to find the $k$-skeletal
object in the coset $\ch+K$ as follows. Starting at $\ch_{ss}$,
which is $(b-1)$-stable and hence $k$-stable, look for a legal borrow
move $\beta_S$ where $S$ meets the conditions in
Lemma~\ref{lem:test-skel}.  If no such move exists, we have reached a
$k$-skeletal configuration.  Otherwise, execute that borrow move to
reach a new $k$-stable configuration and repeat the process. This
series of moves must terminate, since borrow moves increase the total
chip count (Proposition~\ref{prop:fin})
and $k$-stable configurations (being $0$-stable) 
always have fewer than $a$ chips at each vertex.

\bibliography{ratbib}{}
\bibliographystyle{alphaurl}

\end{document}